\documentclass[11pt]{amsart}
\usepackage[british]{babel}

\usepackage{amssymb,latexsym}
\usepackage{comment}
\usepackage{titletoc}
\usepackage{amsfonts, amsmath, amsthm, amssymb, fancybox, graphicx, color, times, babel}
\usepackage[utf8]{inputenc}
\usepackage{multirow}
\usepackage{enumerate}

\usepackage[usenames,dvipsnames]{xcolor}

\usepackage[all,cmtip]{xy}
\usepackage[normalem]{ulem}

\usepackage{tikz-cd}

\newtheorem {thm}{Theorem}
\newtheorem* {thm*}{Theorem}

\newtheorem {cor}[thm]{Corollary}

\newtheorem* {cor*}{Corollary}
\newtheorem {lem}[thm]{Lemma}

\newtheorem {prop}[thm]{Proposition}
\newtheorem* {prop*}{Proposition}

\theoremstyle{definition}
\newtheorem {rem}[thm]{Remark}
\newtheorem {defi}[thm]{Definition}
\newtheorem {exa}[thm]{Example}
\newtheorem* {conj*}{Conjecture}

\newtheorem* {quest*}{Question}

\numberwithin{thm}{section}

\DeclareMathOperator{\Aut}{Aut}

\DeclareMathOperator{\End}{End}

\DeclareMathOperator{\gal}{Gal}
\DeclareMathOperator{\Hom}{Hom}

\DeclareMathOperator{\GL}{GL}
\DeclareMathOperator{\SL}{SL}
\DeclareMathOperator{\PSL}{PSL}
\DeclareMathOperator{\PGL}{PGL}
\DeclareMathOperator{\mat}{Mat}

\DeclareMathOperator{\Id}{Id}
\DeclareMathOperator{\Mat}{Mat}

\DeclareMathOperator{\im}{Im}
\DeclareMathOperator{\ab}{ab}
\DeclareMathOperator{\disc}{disc}

\newcommand{\F}{\mathbb{F}}

\newcommand{\Q}{\mathbb{Q}}
\newcommand{\NN}{\mathbb{N}}
\newcommand{\Z}{\mathbb{Z}}
\renewcommand{\F}{\mathbb{F}}

\newcommand{\Kbar}{\overline{K}}

\newcommand{\lval}{v_\ell}

\newcommand{\order}{\mathcal{A}}

\newcommand{\set}[1]{\left\{ #1 \right\}}

\newcommand{\campo}{J}

\renewcommand{\geq}{\geqslant}
\renewcommand{\leq}{\leqslant}

\makeatletter
\newcommand{\customlabel}[2]{
   \protected@write \@auxout {}{\string \newlabel {#1}{{#2}{\thepage}{#2}{#1}{}} }
   \hypertarget{#1}{}
}
\makeatother

\usepackage{xr-hyper}
\usepackage{hyperref}

\newcommand{\J}{J}
\newcommand{\JJ}{J(2)}

\begin{document}

\author{Davide Lombardo}
\address[]{Dipartimento di Matematica, Università di Pisa, Largo Bruno Pontecorvo 5, 56127 Pisa, Italy}
\email{davide.lombardo@unipi.it}

\title{Explicit Kummer Theory for Elliptic Curves}
\author{Sebastiano Tronto}
\address[]{Mathematics Research Unit, University of Luxembourg, 6 av.\@ de la Fonte, 4364 Esch-sur-Alzette, Luxembourg}
\email{sebastiano.tronto@uni.lu}

%\keywords{Number fields, Kummer theory, Degree.}
%\subjclass[2010]{Primary: 11F80; Secondary: 14L10, 14K15}

\begin{abstract}
Let $E$ be an elliptic curve defined over a number field $K$, let $\alpha \in E(K)$ be a point of infinite order, and let $N^{-1}\alpha$ be the set of $N$-division points of $\alpha$ in $E(\overline{K})$. 
We prove strong effective and uniform results for the degrees of the Kummer extensions $[K(E[N],N^{-1}\alpha) : K(E[N])]$. When $K=\Q$, and under a minimal assumption on $\alpha$, we show that the inequality $[\Q(E[N],N^{-1}\alpha) : \Q(E[N])] \geq cN^2$ holds with a constant $c$ independent of both $E$ and $\alpha$.
\end{abstract}

\maketitle

\section{Introduction}

\subsection{Setting}
Let $E$ be an elliptic curve defined over a number field $K$ (for which we fix an algebraic closure $\Kbar$) and let $\alpha\in E(K)$ be a point of infinite order.
The purpose of this paper is to study the extensions of $K$ generated by the division points of $\alpha$; in order to formally introduce these extensions we need to set some notation.

Given a positive integer $M$, we denote by $E[M]$ the group of $M$-torsion points of $E$, that is, the set $\{P \in E(\Kbar) : MP=0\}$ equipped with the group law inherited from $E$.
Moreover, we denote by $K_M$
the \emph{$M$-th torsion field} $K(E[M])$ of $E$, namely, the finite extension of $K$ obtained by adjoining the coordinates of all the $M$-torsion points of $E$. For each positive integer $N$ dividing $M$, we let $N^{-1}\alpha:=\set{\beta\in E(\Kbar)\mid N\beta=\alpha}$ denote the set of $N$-division points of $\alpha$ and set
\begin{align*}
K_{M,N}:=K(E[M],N^{-1}\alpha).
\end{align*}
The field $K_{M,N}$ is called the \emph{$(M,N)$-Kummer extension} of $K$ (related to $\alpha$), and both $K_M$ and $K_{M,N}$ are finite Galois extensions of $K$.

It is a classical question to study the degree of $K_{M,N}$ over $K_M$ as $M,N$ vary, see for example \cite[Théorème 5.2]{Durham}, \cite[Lemme 14]{Hindry}, or Ribet's foundational paper \cite{MR552524}. In particular, it is known that there exists an integer $C=C(E/K, \alpha)$, depending only on $E/K$ and $\alpha$, such that
\begin{align*}
\frac{N^2}{\left[K_{M,N}:K_M\right]}\quad\text{divides}\quad C
\end{align*}
for every pair of positive integers $(M,N)$ with $N\mid M$.

The aim of this paper is to give an explicit version of this result, and to show that it can even be made uniform when the base field is $K=\Q$. 
Our first result is that, under the assumption $\operatorname{End}_K(E)=\Z$, the integer $C$ can be bounded (explicitly) in terms of the $\ell$-adic Galois representations attached to $E$ and of divisibility properties of the point $\alpha$, and that this statement becomes false if we remove the hypothesis $\operatorname{End}_K(E)=\Z$.

On the other hand, the assumption $\End_\Q(E)=\Z$ is always satisfied when $K=\Q$, and we show that in this case $C$ can be taken to be independent of $E$ and $\alpha$, provided that $\alpha$ and all its translates by torsion points are not divisible by any $n>1$ in the group $E(\Q)$. This is a rather surprising statement, especially given that such a strong uniformity result is not known for the closely connected problem of studying the degrees of the torsion fields $K_M$ over $K$.

\subsection{Main results}
Our main results are the following.

\begin{thm}\label{thm:Main}
Assume that $\End_K(E)=\Z$. There is an explicit constant $C$, depending only on $\alpha$ and on the $\ell$-adic torsion representations associated to $E$ for all primes $\ell$, such that
\begin{align*}
\frac{N^2}{\left[K_{M,N}:K_M\right]}\quad\text{divides}\quad C
\end{align*}
for all pairs of positive integers $(M,N)$ with $N$ dividing $M$.
\end{thm}

The proof gives an explicit expression for $C$ that depends on computable parameters associated with $E$ and $\alpha$. We also show that all these quantities can be bounded effectively in terms of standard invariants of the elliptic curve and of the height of $\alpha$, see Remark \ref{rmk:MainTheoremIsEffective}.

\begin{thm}\label{thm:UniformIntroduction}
There is a universal constant $C$ with the following property. Let $E/\Q$ be an elliptic curve, and let $\alpha \in E(\Q)$ be a point such that the class of $\alpha$ in the free abelian group $E(\Q)/E(\Q)_{\operatorname{tors}}$ is not divisible by any $n>1$. Then
\[
\frac{N^2}{\left[ \Q_{M,N} : \Q_M \right]} \quad \text{ divides } \quad C
\]
for all pairs of positive integers $(M,N)$ with $N$ dividing $M$.
\end{thm}

\subsection{Structure of the paper}

We start with some necessary general preliminaries in Section \ref{sec:Preliminaries}, leading up to a factorisation of the constant $C$ of Theorem \ref{thm:Main} as a product of certain contributions which we dub the \textit{$\ell$-adic} and \textit{adelic} failures (corresponding to $E$, $\alpha$, and a fixed prime $\ell$). In the same section we also introduce some of the main actors of this paper, in the form of several Galois representations associated with the torsion and Kummer extensions. In Section \ref{sec:PropertiesTorsionRepresentation} we then recall some important properties of the torsion representations that will be needed in the rest of the paper.
In Sections \ref{sec:lAdicFailure} and \ref{sec:AdelicFailure} we study the $\ell$-adic and adelic failures respectively. 
In Section \ref{sec:CMCounterexample} we show that one cannot hope to na\"ively generalise some of the results in section \ref{sec:lAdicFailure} to CM curves. Finally, in Section \ref{sec:UniformBounds} we prove Theorem \ref{thm:UniformIntroduction} by establishing several auxiliary results about the Galois cohomology of the torsion modules $E[M]$ that might have an independent interest.

\subsection{Acknowledgements}
It is a pleasure to thank Antonella Perucca for suggesting the problem that led to this paper, for her constant support, and for her useful comments.
We are grateful to Peter Bruin for many interesting discussions, and to Peter Stevenhagen and Francesco Campagna for useful correspondence about the results of section \ref{sec:IntersectionTorsionFields}.

\section{Preliminaries}
\label{sec:Preliminaries}
\subsection{Notation and definitions}
The letter $K$ will always denote a number field, $E$ an elliptic curve defined over $K$, and $\alpha$ a point of infinite order in $E(K)$.

For $n$ a positive integer, we denote by $\zeta_n$ a primitive root of unity of order $n$.

Given a prime $\ell$, we denote by $v_\ell$ the usual $\ell$-adic valuation on $\Q$ and on $\Q_\ell$. If $X$ is a vector in $\Z_\ell^n$ or a matrix in $\Mat_{m \times n}(\Z_\ell)$, we call \textit{valuation} of $X$, denoted by $v_\ell(X)$, the minimum of the $\ell$-adic valuations of its coefficients.

We shall often use divisibility conditions involving the symbols $\ell^\infty$ (where $\ell$ is a prime) and $\infty$. Our convention is that every power of $\ell$ divides $\ell^\infty$, every positive integer divides $\infty$, and $\ell^\infty$ divides $\infty$.

Recall from the Introduction that we denote by $K_M$ the field $K(E[M])$ generated by the coordinates of the $M$-torsion points of $E$, and by $K_{M,N}$ (for $N \mid M$) the field $K(E[M], N^{-1} \alpha)$. We extend this notation by setting $K_{\ell^\infty} = \bigcup_n K_{\ell^n}$, $K_\infty = \bigcup_M K_M$, and more generally, for $M, N \in \mathbb{N}_{>0} \cup \{\ell^\infty, \infty\}$ with $N \mid M$,
\[
K_M = \bigcup_{d \mid M} K_d, \quad K_{M,N} = \bigcup_{d \mid M} \bigcup_{\substack{e \mid d \\ e \mid N}} K_{d,e}
\]

If $H$ is a subgroup of $\GL_2(\Z_\ell)$, we denote by $\Z_\ell[H]$ the sub-$\Z_\ell$-algebra of $\Mat_2(\Z_\ell)$ generated by the elements of $H$.

Let $G$ be a (profinite) group. We write $G'$ for its derived subgroup, namely, the subgroup of $G$ (topologically) generated by commutators, and $G^{\operatorname{ab}}=G/G'$ for its abelianisation, namely, its largest abelian (profinite) quotient.
We say that a finite simple group $S$ \textit{occurs} in a profinite group $G$ if there are closed subgroups $H_1, H_2$ of $G$, with $H_1 \triangleleft H_2$, such that $H_2/H_1$ is isomorphic to $S$. Finally, we denote by $\exp G$ the exponent of a finite group $G$, namely, the smallest integer $e \geq 1$ such that $g^e=1$ for every $g \in G$.

\subsection{The $\ell$-adic and adelic failures}

We start by observing that it is enough to restrict our attention to the case $N=M$:
\begin{rem}
\label{rem-NM}
Suppose that there is a constant $C\geq1$ satisfying
\begin{align*}
\frac{M^2}{\left[K_{M,M}:K_M\right]}\quad\text{divides}\quad C
\end{align*}
for all positive integers $M$. Then for any $N\mid M$, since $[K_{M,M}:K_{M,N}]$ divides $(M/N)^2$, we have that
\begin{align*}
\frac{N^2}{\left[K_{M,N}:K_M\right]}=\frac{N^2[K_{M,M}:K_{M,N}]}{[K_{M,M}:K_M]}\quad \text{divides} \quad \frac{M^2}{\left[K_{M,M}:K_M\right]},
\end{align*}
which in turn divides $C$.
\end{rem}

Elementary field theory gives
\begin{align*}
\frac{N^2}{[K_{N,N}:K_{N}]}=&\prod_{\overset{\ell|N}{\ell\text{ prime}}}\frac{\ell^{2n_\ell}}{[K_{N,\ell^{n_\ell}}:K_N]}=\\
=&\prod_{\overset{\ell|N}{\ell\text{ prime}}}\frac{\ell^{2n_\ell}} {[K_{\ell^{n_\ell},\ell^{n_\ell}}:K_{\ell^{n_\ell}}]}\cdot\frac{[K_{\ell^{n_\ell},\ell^{n_\ell}}:K_{\ell^{n_\ell}}]} {[K_{N,\ell^{n_\ell}}:K_N]}=\\
=&\prod_{\overset{\ell|N}{\ell\text{ prime}}} \frac{\ell^{2n_\ell}} {[K_{\ell^{n_\ell},\ell^{n_\ell}}:K_{\ell^{n_\ell}}]} \cdot [K_{\ell^{n_\ell},\ell^{n_\ell}}\cap K_N:K_{\ell^{n_\ell}}]
\end{align*}
where $n_\ell=\lval(N)$. To see why the first equality holds, recall that the degree $[K_{N, \ell^{n_\ell}} : K_N]$ is a power of $\ell$, so the fields $K_{N, \ell^{n_\ell}}$ are linearly disjoint over $K_N$, and clearly they generate all of $K_{N,N}$.

\begin{defi}
Let $\ell$ be a prime and $N$ a positive integer. Let $n:=\lval(N)$. We call
\begin{align*}
A_\ell(N):=\frac{\ell^{2n}} {[K_{\ell^{n},\ell^{n}}:K_{\ell^{n}}]}
\end{align*}
the \emph{$\ell$-adic failure} at $N$ and
\begin{align*}
B_\ell(N):=\frac{[K_{\ell^{n},\ell^{n}}:K_{\ell^{n}}]} {[K_{N,\ell^{n}}:K_N]}=[K_{\ell^{n},\ell^{n}}\cap K_N:K_{\ell^{n}}]
\end{align*}
the \emph{adelic failure} at $N$ (related to $\ell$). Notice that both $A_\ell(N)$ and $B_\ell(N)$ are powers of $\ell$.
\end{defi}

\begin{exa}
It is clear that the $\ell$-adic failure $A_\ell(N)$ can be nontrivial, that is, different from $1$. Suppose for example that $\alpha=\ell\beta$ for some $\beta\in E(K)$: then we have
\begin{align*}
K_{\ell^n,\ell^n}=K_{\ell^n}(\ell^{-n}\alpha)=K_{\ell^n}(\ell^{-n+1}\beta),
\end{align*}
and the degree of this field over $K_{\ell^n}$ is at most $\ell^{2(n-1)}$, so $\ell^2\mid A_\ell(N)$.
In Example \ref{exa-17739g1} we will show that the $\ell$-adic failure can be non-trivial also when $\alpha$ is strongly $\ell$-indivisible (see Definition \ref{def-divisib}).
\end{exa}

\begin{exa}
We now show that the adelic failure $B_\ell(N)$ can be non-trivial as well. Consider the elliptic curve $E$ over $\Q$ given by the equation
\begin{align*}
y^2 = x^{3} + x^{2} - 44 x - 84
\end{align*}
and with Cremona label 624f2 (see \cite[\href{http://www.lmfdb.org/EllipticCurve/Q/624f2/}{label 624f2}]{lmfdb}). One can show that $E(\Q)\cong \Z\oplus (\Z/2\Z)^2$, so that the curve has full rational $2$-torsion, and that a generator of the free part of $E(\Q)$ is given by $P=(-5,6)$. The $2$-division points of $P$ are given by $(1 + \sqrt{-3}, -3 + 7\sqrt{-3})$, $(-11 + 3\sqrt{-3},27 + 15\sqrt{-3})$, and their Galois conjugates, so they are defined over $\Q(\zeta_3)\subseteq \Q_3$, and we have $B_2(6):=[\Q_{2,2}\cap \Q_6:\Q_2]=[\Q(\zeta_3):\Q]=2$.

These computations have been checked with SageMath \cite{sagemath}.
\end{exa}

\subsection{The torsion, Kummer and arboreal representations}\label{subsec:Representations}
In this section we introduce three representations of the absolute Galois group of $K$ that will be our main tool for studying the extensions $K_{M,N}$. For further information about these representations see for example \cite[Section 3]{JonesRouse},  \cite{2018arXiv180208527B}, and \cite{2016arXiv161202847L}.

\subsubsection{The torsion representation}

Let $N$ be a positive integer. The group $E[N]$ of $N$-torsion points of $E$ is a free $\Z/N\Z$-module of rank $2$. Since the multiplication-by-$N$ map is defined over $K$, the absolute Galois group of $K$ acts $\Z/N\Z$-linearly on $E[N]$, and we get a homomorphism
\begin{align*}
\tau_N:\gal(\Kbar\mid K)\to \Aut(E[N]).
\end{align*}

The field fixed by the kernel of $\tau_N$ is exactly the $N$-th torsion field $K_N$. Thus, after fixing a $\Z/N\Z$-basis of $E[N]$, the Galois group $\gal(K_N\mid K)$ is identified with a subgroup of $\GL_2(\Z/N\Z)$ which we denote by $H_N$. 

As $N$ varies, and provided that we have made compatible choices of bases, these representations form a compatible projective system, so we can pass to the limit over powers of a fixed prime $\ell$ to obtain the \textit{$\ell$-adic torsion representation} $\tau_{\ell^\infty} : \gal(\overline{K} \mid K) \to \GL_2(\Z_\ell)$. We can also take the limit over all integers $N$ (ordered by divisibility) to obtain the \textit{adelic torsion representation} $\tau_\infty : \gal(\overline{K} \mid K) \to \GL_2(\hat{\Z})$. We denote by $H_{\ell^\infty}$ (resp.~$H_\infty$) the image of $\tau_{\ell^\infty}$ (resp.~$\tau_\infty$). The group $H_{\ell^\infty}$ (resp.~$H_{\infty}$) is isomorphic to $\gal(K_{\ell^\infty} \mid K)$ (resp.~$\gal(K_\infty \mid K)$).

One can also pass to the limit on the torsion subgroups themselves, obtaining the \textit{$\ell$-adic Tate module} $T_\ell E=\varprojlim_n E[\ell^n] \cong \Z_\ell^2$ and the \textit{adelic Tate module} $TE = \varprojlim_M E[M] \cong \hat{\Z}^2 \cong \prod_{\ell} \Z_\ell^2$.

\subsubsection{The Kummer representation}

Let $M$ and $N$ be positive integers with $N\mid M$. Let $\beta\in E(\Kbar)$ be a point such that $N\beta=\alpha$. For any $\sigma\in\gal(\Kbar\mid K_M)$ we have that $\sigma(\beta)-\beta$ is an $N$-torsion point, so the following map is well-defined:
\[
\begin{array}{cccc}
\kappa_N: & \gal(\Kbar\mid K_M) &\to & E[N]\\
& \sigma &\mapsto & \sigma(\beta)-\beta.
\end{array}
\]
Since any other $N$-division point $\beta'$ of $\alpha$ satisfies $\beta'=\beta+T$ for some $T\in E[N]$, and the coordinates of $T$ belong to $K_N \subseteq K_M$, the map $\kappa_N$ does not depend on the choice of $\beta$. It is also immediate to check that $\kappa_N$ is a group homomorphism, and that the field fixed by its kernel is exactly the $(M,N)$-Kummer extension of $K$. Fixing a basis of $E[N]$ we can identify the Galois group $\gal(K_{M,N}\mid K_M)$ with a subgroup of $(\Z/N\Z)^2$. It is then clear that $K_{M,N}$ is an abelian extension of $K_M$ of degree dividing $N^2$, and the Galois group of this extension has exponent dividing $N$.

In the special case $M=N$ we denote by $V_N$ the image of $\gal\left( K_{N,N} \mid K_N \right)$ in $(\Z/N\Z)^2$.

By passing to the limit in the previous constructions we also obtain the following:
\begin{itemize}
\item There is an $\ell$-adic Kummer representation $\kappa_{\ell^\infty} : \gal(\overline{K} \mid K_{\ell^\infty}) \to T_\ell E$ which factors via a map $\gal(K_{\ell^\infty, \ell^\infty} \mid K_{\ell^\infty}) \to T_\ell E$ (still denoted by $\kappa_{\ell^\infty}$). 
\item The image $V_{\ell^\infty}$ of $\kappa_{\ell^\infty}$ is a sub-$\Z_\ell$-module of $T_\ell E \cong \Z_\ell^2$, isomorphic to $\gal(K_{\ell^\infty, \ell^\infty} \mid K_{\ell^\infty})$ as a profinite group. We therefore identify $\gal(K_{\ell^\infty, \ell^\infty} \mid K_{\ell^\infty})$ with $V_{\ell^\infty}$.
\item We can identify the Galois group $\gal(K_{\infty, \ell^\infty} \mid K_\infty)$ with a $\Z_\ell$-submodule $W_{\ell^\infty}$ of $V_{\ell^\infty}$ (hence also of $T_\ell E$) via the representation $\kappa_{\ell^\infty}$.
\item We can identify the Galois group $\gal(K_{\infty, \infty} \mid K_\infty)$ with a sub-$\hat{\Z}$-module $W_\infty$ of $TE \cong \hat{\Z}^2$.
\end{itemize}

Notice that $W_{\ell^\infty}$ is the projection of $W_{\infty}$ in $\Z_\ell^2$, and since $W_{\ell^\infty}$ is a pro-$\ell$ group and there are no nontrivial continuous morphisms from a pro-$\ell$ group to a pro-$\ell'$ group for $\ell \neq \ell'$ we have $W_\infty = \prod_\ell W_{\ell^\infty}$.

\subsubsection{The arboreal representation}
Fix a sequence $\{\beta_i\}_{i\in\NN}$ of points in $E(\overline{K})$ such that $\beta_1=\alpha$ and $N \beta_M=\beta_{M/N}$ for all pairs of positive integers $(N,M)$ with $N \mid M$. For every $N\geq 1$ fix furthermore a $\Z/N\Z$-basis $\{T_1^N,T_2^N\}$ of $E[N]$ in such a way that $N T_1^M=T_1^{M/N}$ and $N T_2^M=T_2^{M/N}$ for every pair of positive integers $(N,M)$ with $N \mid M$. For every $N \geq 1$, the map
\begin{align*}
\omega_N:\gal(K_{N,N}\mid K)&\to \left(\Z/N\Z\right)^2\rtimes\GL_2\left(\Z/N\Z\right)\\
\sigma&\mapsto \left(\sigma(\beta_N)-\beta_N,\tau_{N}(\sigma)\right)
\end{align*}
is an injective homomorphism (similarly to \cite[Proposition 3.1]{JonesRouse}) and thus identifies the group $\gal(K_{N,N}\mid K)$ with a subgroup of $\left(\Z/N\Z\right)^2\rtimes\GL_2\left(\Z/N\Z\right)$.

It will be important for our applications to notice that $V_N$ comes equipped with an action of $H_N$ coming from the fact that $V_N$ is the (abelian) kernel of the natural map $\gal(K_{N,N}\mid K) \to H_N$. More precisely, the action of $h \in H_N$ on $v \in V_N$ is given by conjugating the element $(v, \operatorname{Id}) \in (\Z/N\Z)^2 \rtimes \GL_2(\Z/N\Z)$ by $(0,h)$. Explicitly, we have
\[
(0,h) (v, \operatorname{Id}) (0,h)^{-1} = (hv,h)(0,h^{-1}) = (hv, \operatorname{Id}),
\]
so that the action of $H_N$ on $V_N$ is induced by the natural action of $\GL_2(\Z/N\Z)$ on $\left(\Z/N\Z \right)^2$. We obtain similar statements by suitably passing to the limit in $N$:

\begin{lem}\label{lemma:HnActionOnVn}
For every positive integer $N$, the group $V_N$ is an $H_N$-submodule of $(\Z/N\Z)^2$ for the natural action of $H_N \leq \GL_2(\Z/N\Z)$ on $V_N \leq (\Z/N\Z)^2$. Similarly, both $V_{\ell^\infty}$ and $W_{\ell^\infty}$ are $H_{\ell^\infty}$-modules.
\end{lem}

\begin{rem}\label{rem:Vsubgroup}
Let $N \in \mathbb{N} \cup \{\ell^\infty\}$ and $M \in \mathbb{N} \cup \{ \ell^\infty, \infty \}$ with $N \mid M$. Then $\gal(K_{M,N} \mid K_M)$ can be identified with a subgroup of $V_N$: this follows from inspection of the diagram
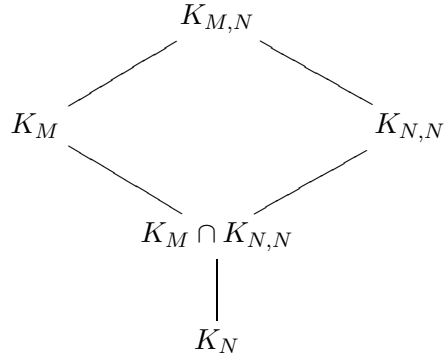
\[
\xymatrix{
& K_{M,N} \ar@{-}[dr] \ar@{-}[dl] \\
K_M \ar@{-}[dr] && K_{N,N} \ar@{-}[dl]  \\
& K_M \cap K_{N,N} \ar@{-}[d] \\
& K_N
}
\]
which shows that $\gal(K_{M,N} \mid K_M)$ is isomorphic to $\gal( K_{N,N} \mid K_M \cap K_{N,N})$, which in turn is clearly a subgroup of $\gal(K_{N,N} \mid K_N) \cong V_N$.
\end{rem}

\subsection{Curves with complex multiplication}\label{subsect:CMCurves}

If $\End_{\Kbar}(E)\neq \mathbb{Z}$ we say that \emph{$E$ has complex multiplication}, or CM for short. In this case $\End_{\Kbar}(E)$ is an order in an imaginary quadratic field, called \emph{the CM-field of $E$}. The torsion representations in the CM case have been studied for example in \cite{Deur1} and \cite{Deur2}.

In this case, the image of the torsion representation $\tau_{\ell^\infty}$ is closely related to the \textit{Cartan subgroup of $\operatorname{GL}_2(\Z_\ell)$ corresponding to $\operatorname{End}_{\overline{K}}(E)$}, defined as follows:
\begin{defi}
Let $F$ be a reduced $\Q_\ell$-algebra of degree $2$ and let $\order_\ell$ be a $\Z_\ell$-order in $F$. 
The \textit{Cartan subgroup} corresponding to $\order_\ell$ is the group of units of $\order_\ell$, which we embed in $\GL_2(\Z_\ell)$ by fixing a $\Z_\ell$-basis of $\order_\ell$ and considering the left multiplication action of $\order_\ell^\times$. If $\order$ is an order in an imaginary quadratic number field, the Cartan subgroup of $\GL_2(\Z_\ell)$ corresponding to $\order$ is defined by taking $\order_\ell=\order \otimes \Z_\ell$ in the above.
\end{defi}

More precisely, when $E/K$ is an elliptic curve with CM, the image of the $\ell$-adic torsion representation $\tau_{\ell^\infty}$ is always contained (up to conjugacy in $\GL_2(\Z_\ell)$) in the normaliser of the Cartan subgroup corresponding to $\operatorname{End}_{\overline{K}}(E)$, and is contained in the Cartan subgroup itself if and only if the complex multiplication is defined over the base field $K$.

In order to have a practical representation of Cartan subgroups, we recall the following definition from \cite{MR3690236}:
\begin{defi}\label{def:CartanParameters}
Let $C$ be a Cartan subgroup of $\GL_2(\Z_\ell)$. We say that $(\gamma, \delta) \in \Z_\ell^2$ are \textit{parameters for $C$} if $C$ is conjugated in $\GL_2(\Z_\ell)$ to the subgroup 
\begin{equation}\label{nfaa}
\left\{ \begin{pmatrix}
x & \delta y \\ y & x+\gamma y
\end{pmatrix} : x,y \in \Z_{\ell},\; v_{\ell}(x(x+\gamma y)-\delta y^2)=0 \right\}\,.
\end{equation}
Parameters for $C$ always exist, see \cite[§2.3]{MR3690236}.
\end{defi}
\begin{rem}[{\cite[Remark 9]{MR3690236}}]\label{rem:Parameters}
One may always assume that $\gamma,\delta$ are integers. Furthermore, one can always take $\gamma \in \{0,1\}$, and $\gamma=0$ if $\ell \neq 2$.
\end{rem}

We also recall the following explicit description of the normaliser of a Cartan subgroup \cite[Lemma 14]{MR3690236}:
\begin{lem}\label{lem-Norm}
A Cartan subgroup has index 2 in its normaliser. If $C$ is as in \eqref{nfaa}, its normaliser $N$ in $\GL_2(\Z_\ell)$ is the disjoint union of $C$ and  
$\displaystyle C':=\begin{pmatrix}
1 & \gamma \\ 0 & -1
\end{pmatrix} \cdot C\,.$
\end{lem}

\section{Properties of the torsion representation}
\label{sec:PropertiesTorsionRepresentation}

Torsion representations are studied extensively in the literature; we have in particular the following fundamental theorem of Serre \cite{Serre}, which applies to all elliptic curves (defined over number fields) without complex multiplication:

\begin{thm}[Serre]\label{thm:Serre}
If $\End_{\Kbar}(E)=\Z$, then $H_\infty$ is open in $\GL_2(\hat{\Z})$.
Equivalently, the index of $H_N$ in $\GL_2(\Z/N\Z)$ is bounded independently of $N$.
\end{thm}

There is also a CM analogue of Theorem \ref{thm:Serre}, which is more easily stated by introducing the following definition:

\begin{defi}\label{def:MaximalRepresentation}
Let $E/K$ be an elliptic curve and $\ell$ be a prime number. We say that the image of the $\ell$-adic representation is \textit{maximal} if one of the following holds:
\begin{itemize}
\item $E$ does not have CM over $\overline K$ and $H_{\ell^\infty}=\GL_2(\Z_\ell)$.
\item $E$ has CM over $K$ by an order $\order$ in the imaginary quadratic field $F$, the prime $\ell$ is unramified in $F$ and does not divide $[\mathcal{O}_F : \order]$, and $H_{\ell^\infty}$ is conjugated to the Cartan subgroup of $\GL_2(\Z_\ell)$ corresponding to $\order$.
\item $E$ has CM over $\overline K$ (but not over $K$) by an order $\order$ in the imaginary quadratic field $F$, the prime $\ell$ is unramified in $F$ and does not divide $[\mathcal{O}_F : \order]$, and $H_{\ell^\infty}$ is conjugated to the normaliser of the Cartan subgroup of $\GL_2(\Z_\ell)$ corresponding to $\order$.
\end{itemize}
\end{defi}

\begin{thm}[{\cite[Corollaire on p.302]{Serre}}]\label{thm:CMSerre}
Let $E/K$ be an elliptic curve admitting CM over $\overline{K}$. Then the $\ell$-adic representation attached to $E/K$ is maximal for all but finitely many primes $\ell$.
\end{thm}

In the rest of this section we recall various important properties of the torsion representations: we shall need results that describe both the asymptotic behaviour of the $\bmod \,\ell^n$ torsion representation as $n \to \infty$ (§\ref{subsect:MaximalGrowth} and \ref{subsect:UniformGrowth}) and the possible images of the $\bmod\,\ell$ representations attached to elliptic curves defined over the rationals (§\ref{subsec:PossibleImagesModl}).

\subsection{Maximal growth}\label{subsect:MaximalGrowth}

We recall some results on the growth of the torsion extensions from \cite[§2.3]{2016arXiv161202847L}.

\begin{prop}
\label{prop-max-growth}
Let $\ell$ be a prime number. Let $\delta=2$ if $E$ has complex multiplication and $\delta=4$ otherwise. There exists a positive integer $n_\ell$ such that
\begin{align*}
\#H_{\ell^{n+1}}/\#H_{\ell^n}=\ell^\delta \qquad \text{for every }n\geq n_\ell.
\end{align*}
\end{prop}
\begin{proof}
This follows from Theorem \ref{thm:Serre} in the non-CM case and from classical results in the CM case. See also \cite[Lemma 10 and Remark 13]{2016arXiv161202847L} for a more general result.
\end{proof}

\begin{defi}
We call an integer $n_\ell$ as in Proposition \ref{prop-max-growth} a \emph{parameter of maximal growth for the $\ell$-adic torsion representation}. We say that it is \emph{minimal} if $n_\ell-1$ is not a parameter of maximal growth; when $\ell=2$, we require that the minimal parameter be at least 2.
\end{defi}

\begin{rem}
The assumption $n_\ell \geq 2$ when $\ell=2$ will be needed to apply \cite[Theorem 12]{2016arXiv161202847L}.
\end{rem}

\begin{rem}\label{rmk:nlIsEffective} Given an explicit elliptic curve $E/K$ and a prime $\ell$, the problem of determining the optimal value of $n_\ell$ can be solved effectively (see \cite[Remark 13]{2016arXiv161202847L}). However, computing $n_\ell$ can be challenging in practice, because the na\"ive algorithm requires the determination of the Galois groups of the splitting fields of several large-degree polynomials. The situation is usually better for smaller primes $\ell$, and especially for $\ell=2$, for which the 2-torsion tower is known essentially explicitly (see \cite{MR3500996} for a complete classification result when $K=\mathbb{Q}$, and \cite{MR3426175} for a description of the 2-torsion tower of a given elliptic curve over a number field).
\end{rem}

In Section \ref{sec:AdelicFailure} we will need to bound the minimal parameter of maximal growth for the $\ell$-adic torsion representation defined over certain extensions of the base field. We will do so with the help of the following Lemma:

\begin{lem}
\label{lem-param-tilde-bound}
Let $\tilde K$ be a finite extension of $K$. Let $n_\ell$ (resp.~$\tilde n_\ell$) be the minimal parameter of maximal growth for the $\ell$-adic torsion representation attached to $E/K$ (resp.~$E/\tilde K$).
Then $\tilde n_\ell\leq n_\ell+v_\ell([\tilde K:K])$.
\end{lem}
\begin{proof}

Let $n_0:=n_\ell+v_\ell([\tilde K:K])+1$ and consider the following diagram:

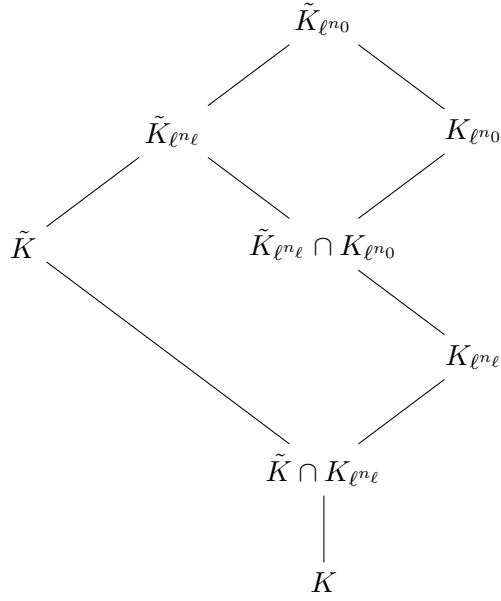
\vspace{0.3cm}

\begin{center}\begin{tikzpicture}[yscale=1.5]
\node (L2tilde) at (0,1) {$\tilde K_{\ell^{n_0}}$};
\node (L1tilde) at (-2,0) {$\tilde K_{\ell^{n_\ell}}$};
\node (L2) at (2,0) {$K_{\ell^{n_0}}$};
\node (L1tildecapL2) at (0,-1) {$\tilde K_{\ell^{n_\ell}}\cap K_{\ell^{n_0}}$};
\node (Ktilde) at (-4,-1) {$\tilde K$};
\node (fantasma) at (4,-1) {\phantom{$\tilde K$}};
\node (L1) at (2,-2) {$K_{\ell^{n_\ell}}$};
\node (KtildecapL1) at (0,-3) {$\tilde K \cap K_{\ell^{n_\ell}}$};
\node (K) at (0,-4) {$K$};
\draw (L2)--(L2tilde)--(L1tilde)--(Ktilde);
\draw (L1)--(L1tildecapL2)--(L1tilde);
\draw (K)--(KtildecapL1)--(Ktilde);
\draw (L2)--(L1tildecapL2);
\draw (L1)--(KtildecapL1);
\end{tikzpicture}\end{center}
Since clearly $[\tilde K_{\ell^{n_\ell}}\cap K_{\ell^{n_0}}:K_{\ell^{n_\ell}}]$ divides $[\tilde K_{\ell^{n_\ell}}:K_{\ell^{n_\ell}}]$, which in turn divides $[\tilde K:K]$, and since $[\tilde K_{\ell^{n_0}}:\tilde K_{\ell^{n_\ell}}]=[K_{\ell^{n_0}}:\tilde K_{\ell^{n_\ell}}\cap K_{\ell^{n_0}}]$, we have
\[
\begin{aligned}
v_\ell\left([K_{\ell^{n_0}}:K_{\ell^{n_\ell}}]\right) & = v_\ell\left([K_{\ell^{n_0}}:\tilde K_{\ell^{n_\ell}} \cap K_{\ell^{n_0}}] \right) + v_\ell\left( [\tilde K_{\ell^{n_\ell}}\cap K_{\ell^{n_0}}:K_{\ell^{n_\ell}}]\right) \\ &\leq v_\ell\left([\tilde K_{\ell^{n_0}}:\tilde K_{\ell^{n_\ell}}]\right)+ v_\ell\left( [\tilde K:K]\right).
\end{aligned}
\]
By \cite[Theorem 12]{2016arXiv161202847L} we have $v_\ell\left([K_{\ell^{n_0}}:K_{\ell^{n_\ell}}]\right)=\delta(n_0-n_\ell)=\delta\left( v_\ell\left([\tilde K:K]\right)+1\right)$, where $\delta$ is as in Proposition \ref{prop-max-growth}, and we get
\begin{align*}
v_\ell\left([\tilde K_{\ell^{n_0}}:\tilde K_{\ell^{n_\ell}}]\right)\geq \delta+ (\delta-1)v_\ell\left([\tilde K:K]\right)> (\delta-1)(n_0-n_\ell).
\end{align*}
Consider now the tower of extensions
\begin{align*}
\tilde K_{\ell^{n_\ell}}\subseteq \tilde K_{\ell^{n_\ell+1}}\subseteq \cdots \subseteq \tilde K_{\ell^{n_0}}
\end{align*}
and notice that by the pigeonhole principle for at least one $n\in \set{n_\ell, n_\ell+1,\dots, n_0-1}$ we must have $[\tilde K_{\ell^{n+1}}:\tilde K_{\ell^{n}}]\geq \delta$. But then by \cite[Theorem 12]{2016arXiv161202847L} we have maximal growth over $\tilde K$ from $n< n_0$. Thus we get $\tilde n_\ell\leq n_\ell+v_\ell\left([\tilde K:K]\right)$ as claimed.
\end{proof}

\subsection{Uniform growth of $\ell$-adic representations}\label{subsect:UniformGrowth}

The results in this subsection and the next will be needed in Section \ref{sec:UniformBounds}.
We start by recalling the following result, due to Arai:
\begin{thm}[{\cite[Theorem 1.2]{MR2434156}}]\label{thm:Arai}
Let $K$ be a number field and let $\ell$ be a prime. Then there exists an integer $n \geq 0$, depending only on $K$ and $\ell$, such that for any elliptic
curve $E$ over $K$ with no complex multiplication over $\overline{K}$ we have
\[
\tau_{\ell^\infty}(\gal(\overline{K} \mid K)) \supseteq \{M \in \GL_2(\Z_\ell) : M \equiv \Id \pmod{\ell^n}\}.
\]
\end{thm}

For the next result we shall need a well-known Lemma about twists of elliptic curves:
\begin{lem}\label{lemma:Untwisting}
Let $E_1, E_2$ be elliptic curves over $K$ such that $(E_1)_{\overline{\Q}}$ is isomorphic to $(E_2)_{\overline{\Q}}$. There is an extension $F$ of $K$, of degree dividing $12$, such that $E_1$ and $E_2$ become isomorphic over $F$.
\end{lem}
\begin{proof}
Fixing a $\overline{\Q}$-isomorphism between $E_1$ and $E_2$ allows us to attach to $E_2$ a class in $H^1\left( \gal(\overline{K} \mid K), \Aut(E_1) \right)$. Since 
$H^1\left( \gal(\overline{K} \mid K), \Aut(E_1) \right) \cong K^\times/K^{\times n}$ for some $n \in \{2, 4, 6\}$ (see \cite[Proposition X.5.4]{SilvermanEC}), the class of $E_2$ corresponds to the class of a certain $[\alpha] \in K^\times/K^{\times n}$. Letting $F=K(\sqrt[n]{\alpha})$, whose degree over $K$ divides $12$, it is clear that $[\alpha] \in F^\times/F^{\times n}$ is the trivial class, so the same is true for $[E_2] \in H^1(\gal(\overline{F} \mid F), \Aut(E_1))$, which means that $E_2$ is isomorphic to $E_1$ over $F$ as desired.
\end{proof}

\begin{cor}\label{cor:UniversalBoundGrowthParameter}
Let $K$ be a number field and $\ell$ be a prime number. There exists an integer $n_\ell$ with the following property: for every elliptic curve $E/K$, the minimal parameter of maximal growth for the $\ell$-adic representation attached to $E$ is at most $n_\ell$.
\end{cor}
\begin{proof}
Let $n$ be the integer whose existence is guaranteed by Theorem \ref{thm:Arai}. By the general theory of CM elliptic curves, we know that there are finitely many values $j_1,\ldots,j_k \in \overline{\Q}$ such that for every CM elliptic curve $E/K$ we have $j(E) \in \{j_1,\ldots,j_k\}$. For each such $j_i$, fix an elliptic curve $E_i/K$ with $j(E_i)=j_i$. To every $E_i/K$ corresponds a minimal parameter of maximal growth for the $\ell$-adic representation that we call $m_i$. Let $n_\ell=\max\{n, m_i+2 \bigm\vert i =1,\ldots,k\}$: we claim that this value of $n_\ell$ satisfies the conclusion of the Corollary. Indeed, let $E/K$ be any elliptic curve. If $E$ does not have CM, the minimal parameter of maximal growth for its $\ell$-adic representation is at most $n \leq n_\ell$. If $E$ has CM, then there exists $i$ such that $j(E)=j_i=j(E_i)$, so $E$ is a twist of $E_i$. By Lemma \ref{lemma:Untwisting} the curves $E$ and $E_i$ become isomorphic over an extension $F/K$ of degree dividing $12$, so if $m$ (resp.~$\tilde{m}$, resp.~$\tilde{m}_i$) denotes the minimal parameter of maximal growth for $E/K$ (resp.~for $E/F$, resp.~for $E_i/F$) we have
\[
m \leq \tilde{m} = \tilde{m}_i \leq m_i + 2 \leq n_\ell,
\]
where the equality follows from the fact that $E$ and $E_i$ are isomorphic over $F$, while the inequality $\tilde{m}_i \leq m_i + 2$ follows from  Lemma \ref{lem-param-tilde-bound} and the fact that $v_\ell([F:K]) \leq v_\ell(12) \leq 2$ for every prime $\ell$.
\end{proof}

\subsection{Possible images of $\bmod\,\ell$ representations}
\label{subsec:PossibleImagesModl}

We recall several results concerning the images of the $\bmod\,\ell$ representations attached to elliptic curves over $\Q$.

We begin with a famous Theorem of Mazur. Let $\mathcal{T}_0:=\set{p\text{ prime }\mid p\leq 17}\cup\{37\}$.

\begin{thm}[{\cite[Theorem 1]{Mazur}}]
\label{thm:Mazur}
Let $E/\Q$ be an elliptic curve and assume that $E$ has a $\mathbb{Q}$-rational subgroup of order $p$. Then $p\in \mathcal T_0\cup \set{19,43,67,163}$. If $E$ does not have CM over $\overline \Q$, then $p\in\mathcal{T}_0$.
\end{thm}

\begin{thm}[{\cite[Proposition 1.13]{PossibleImages}}]
\label{thm:Zywyna}
Let $E/\Q$ be a non-CM elliptic curve and $p\not \in \mathcal{T}_0$ be a prime. Let $C_{\operatorname{ns}}(p)$ be the subgroup of $\GL_2(\F_p)$ consisting of all matrices of the form $\begin{pmatrix}
a & b\epsilon\\
b & a
\end{pmatrix}$ with $(a,b)\in\F_p^2\setminus \{(0,0)\}$ and $\epsilon$ a fixed element of $\F_p^\times\setminus \F_p^{\times 2}$.
Then $H_p$ is conjugate to one of the following:
\begin{enumerate}[(i)]
\item $\GL_2(\F_p)$;
\item the normaliser $N_{\operatorname{ns}}(p)$ of $C_{\operatorname{ns}}(p)$;
\item the index $3$ subgroup
\[
D(p):=\set{a^3\mid a\in C_{\operatorname{ns}}(p)}\cup \set{\begin{pmatrix}
1 & 0\\
0 & -1
\end{pmatrix}\cdot a^3\mid a\in C_{\operatorname{ns}}(p)}
\]
of $N_{\operatorname{ns}}(p)$.
\end{enumerate} 
Moreover, the last case can only occur if $p\equiv 2\pmod 3$.
\end{thm}

\begin{cor}
\label{cor:ContainsScalarsAndConjugation}
Let $E/\Q$ be a non-CM elliptic curve and $p\not \in \mathcal{T}_0$ be a prime. The following hold:
\begin{enumerate}[(1)]
\item The image $H_p$ of the modulo-$p$ representation attached to $E$ contains $\{\lambda\Id\mid \lambda\in\F_p^\times\}$.
\item Suppose $H_p\neq \GL_2(\F_p)$ and let $g_p\in \GL_2(\F_p)$ be an element that normalises $H_p$. Then there is $h\in \GL_2(\F_p)$ such that $h^{-1}g_ph\in N_{\operatorname{ns}}(p)$ and $h^{-1}H_ph\subseteq N_{\operatorname{ns}}(p)$.
\end{enumerate}
\end{cor}
\begin{proof}
\begin{enumerate}[(1)]
\item We apply Theorem \ref{thm:Zywyna}. If $H_p$ is either $\GL_2(\F_p)$ or conjugate to $N_{\operatorname{ns}}(p)$, the conclusion follows trivially, since $C_{\operatorname{ns}}(p)$ contains all scalars. In case (iii) of Theorem \ref{thm:Zywyna}, $H_p$ contains the cubes of the scalars, hence all scalars since $p\equiv 2\pmod 3$.
\item We only have to consider cases (ii) and (iii) of Theorem \ref{thm:Zywyna}. Up to conjugation, we may assume that $H_p\subseteq N_{\operatorname{ns}}(p)$ and the claim becomes $g_p\in N_{\operatorname{ns}}(p)$.

In case (ii) it suffices to check that the normaliser of $N_{\operatorname{ns}}(p)$ is $N_{\operatorname{ns}}(p)$ itself. This holds because $C_{\operatorname{ns}}(p)$, being the only cyclic subgroup of index $2$ of $N_{\operatorname{ns}}(p)$, is characteristic in $N_{\operatorname{ns}}(p)$; hence any element that normalises $N_{\operatorname{ns}}(p)$ normalises $C_{\operatorname{ns}}(p)$ as well, so it must be in $N_{\operatorname{ns}}(p)$.
In case (iii), one similarly sees that $\{a^3\mid a\in C_{\operatorname{ns}}(p)\}$ is characteristic in $D(p)$ and that its normaliser is $N_{\operatorname{ns}}(p)$, and the conclusion follows as above.
\end{enumerate}
\end{proof}

\begin{lem}
\label{lem:LiftHomothety}
Let $\ell$ be a prime number and let $H$ be a closed subgroup of $\GL_2(\Z_\ell)$. Denote by $H_\ell$ the reduction of $H$ modulo $\ell$ and suppose that $H_\ell$ contains a scalar matrix $\overline{\lambda} \operatorname{Id}$. Then $H$ contains a scalar matrix $\lambda \operatorname{Id}$ for some $\lambda \in \Z_\ell^\times$ with $\lambda \equiv \overline{\lambda} \pmod{\ell}$.
\end{lem}
\begin{proof}
Let $h \in H$ be any element that is congruent modulo $\ell$ to $\overline{\lambda} \operatorname{Id}$. Let $\lambda \in \Z_\ell^\times$ be the Teichm\"uller lift of $\overline{\lambda}$ (that is, $\lambda^\ell=\ell$ and $\lambda \equiv \overline{\lambda} \pmod{\ell}$) and write $h=\lambda h_1$, where $h_1 =\Id+\ell A$ for some $A\in\Mat_2(\Z_\ell)$.
The sequence $h^{\ell^n} = \lambda^{\ell^n} h_1^{\ell^n}= \lambda h_1^{\ell^n}$ converges to $\lambda \operatorname{Id}$, because for every $n$ we have $h_1^{\ell^n} = \left( \operatorname{Id} + \ell A \right)^{\ell^n} \equiv \operatorname{Id} \pmod{\ell^n}$. As $H$ is closed, the limit of this sequence, namely $\lambda \Id$, also belongs to $H$ as claimed.
\end{proof}

\section{The $\ell$-adic failure}
\label{sec:lAdicFailure}
The aim of this section is to study the $\ell$-adic failure $A_\ell(N)$ for a fixed prime $\ell$.
The divisibility properties of $\alpha$ in the group $E(K)$ play a crucial role in the study of this quantity, so we begin with the following definition:

\begin{defi}
\label{def-divisib}
Let $\alpha \in E(K)$ and let $n$ be a positive integer. We say that $\alpha$ is \emph{$n$-indivisible over $K$} if there is no $\beta \in E(K)$ such that $n \beta=\alpha$; otherwise we say that $\alpha$ is \textit{$n$-divisible} or \textit{divisible by $n$ over $K$}. Let $\ell$ be a prime number. We say that $\alpha$ is \emph{strongly $\ell$-indivisible over $K$} if the point $\alpha+T$ is $\ell$-indivisible over $K$ for every torsion point $T\in E(K)$ of $\ell$-power order. Finally, we say that $\alpha$ is \emph{strongly indivisible over $K$} if its image in the free abelian group $E(K)/E(K)_{\operatorname{tors}}$ is not divisible by any $n>1$, or equivalently if $\alpha$ is strongly $\ell$-indivisible over $K$ for every prime $\ell$.
\end{defi}

Our aim is to give an analogue of the following result, which bounds the index of the image of the Kummer representation, in those cases when the torsion representation is not surjective.

\begin{thm*}[Jones-Rouse, {\cite[Theorem 5.2]{JonesRouse}}]\label{thm:JonesRouseSurjectivity}
Assume that the $\ell$-adic torsion representation $\tau_{\ell^\infty}:\gal(K_{\ell^\infty}\mid K)\to \GL_2(\Z_\ell)$ is surjective. Assume that $\alpha$ is $\ell$-indivisible in $E(K)$ and, if $\ell=2$, that $K_{2,2}\not\subseteq K_4$. Then the $\ell$-adic Kummer representation $\kappa_{\ell^\infty}:\gal(K_{\ell^\infty,\ell^\infty}\mid K_{\ell^\infty})\to \Z_\ell^2$ is surjective.
\end{thm*}

\subsection{An exact sequence}

We shall need to understand the divisibility properties of $\alpha$ not only over the base field $K$, but also over the division fields of $E$. Thus we turn to studying how the divisibility of the point $\alpha$ by powers of $\ell$ changes when passing to a field extension.
Our main tool will be the following Lemma.

\begin{lem}
\label{lem-divibility-1}
Let $L$ be a finite Galois extension of $K$ with Galois group $G$. For every $m\geq 1$ there is an exact sequence of abelian groups
\begin{align*}
0 \to m E(K)\to E(K)\cap m E(L)\to H^1(G,E[m](L)),
\end{align*}
where the injective map on the left is the natural inclusion.
\begin{proof}
Consider the short exact sequence of $G$-modules
\begin{align*}
0\to E[m](L)\to E(L) \xrightarrow{[m]} m E(L)\to 0
\end{align*}
and the beginning of the long exact sequence in cohomology,
\begin{align*}
0\to (E[m](L))^G\to (E(L))^G\to (m E(L))^G\to H^1(G,E[m](L)).
\end{align*}
Noticing that
\begin{align*}
(E[m](L))^G= E[m](K), && (E(L))^G=E(K), && (m E(L))^G = E(K)\cap m E(L)
\end{align*}
and that
\begin{align*}
E(K)/E[m](K) \cong m E(K)
\end{align*}
the lemma follows.
\end{proof}
\end{lem}

The quotient $\left(E(K)\cap m E(L)\right)/m E(K)$ gives a measure of ``how many'' $K$-points of $E$ are $m$-divisible in $E(L)$ but not $m$-divisible in $E(K)$. 
 We shall often use this Lemma in the special case of $m=\ell^n$ being a power of $\ell$: in this context, the quotient $\left(E(K)\cap \ell^n E(L)\right)/\ell^n E(K)$ is a subgroup of $E(K)/\ell^n E(K)$, so it has exponent dividing $\ell^n$. 
We conclude that if $\ell\nmid \#H^1(G,E[\ell^n](L))$ then no $\ell$-indivisible $K$-point of $E$ can become $\ell$-divisible in $E(L)$. This applies in particular when $\ell\nmid \#G$, see \cite[Proposition 1.6.2]{MR2392026}.

\subsection{Divisibility in the $\ell$-torsion field}

As an example, we investigate the situation of Lemma \ref{lem-divibility-1} with $m=\ell$ and $L=K_\ell$. In this case the exact sequence becomes
\begin{align*}
0 \to \ell E(K)\to E(K)\cap\ell E(K_{\ell})\to H^1(H_\ell,E[\ell]).
\end{align*}

The following Lemma can also be found in \cite[Section 3]{LawsonWutrich}.
\begin{lem}
The cohomology group $H^1(H_\ell,E[\ell])$ is either trivial or cyclic of order $\ell$. When $\ell=2$ it is always trivial.
\begin{proof}
Since $\ell E[\ell]=0$, we have $\ell H^1(H_\ell,E[\ell])=0$. It follows from \cite[Theorem IX.4]{Serre-LocalFields} that we have an injective map
$H^1(H_\ell,E[\ell])\to H^1(S_\ell, E[\ell])$, where $S_\ell$ is an $\ell$-Sylow subgroup of $H_\ell$. This is either trivial, in which case $H^1(H_\ell,E[\ell])=0$, or cyclic of order $\ell$. In the latter case, up to a change of basis for $E[\ell]$ we can assume that $S_\ell$ is generated by $\sigma=\left(\begin{array}{cc}1&1\\ 0&1\end{array}\right)$. One can conclude the proof by explicitly computing the cohomology of the cyclic group $\langle \sigma\rangle$ as in \cite[Lemma 7]{LawsonWutrich}.
\end{proof}
\end{lem}

In \cite{LawsonWutrich} the authors classify the cases when $H^1(H_\ell,E[\ell]) \neq 0$ for $K=\Q$ and they give rather complete results in case $K$ is a number field with $K\cap\Q(\zeta_\ell)=\Q$. In particular, it turns out that, for $K=\Q$, the group
$H^1(H_\ell,E[\ell])$ can be non-trivial only when $\ell=3,5,11$, and only when additional conditions are satisfied (see \cite[Theorem 1]{LawsonWutrich}).

The next Example shows that for $K=\Q$ a point in $E(\Q)$ that is strongly $3$-indivisible may become $3$-divisible over the $3$-torsion field.

\begin{exa}
\label{exa-17739g1}
Consider the elliptic curve $E$ over $\Q$ given by the equation
\begin{align*}
y^2 + y = x^{3} - 216 x - 1861  
\end{align*}
with Cremona label 17739g1 (see \cite[\href{http://www.lmfdb.org/EllipticCurve/Q/17739g1/}{label 17739g1}]{lmfdb}).
We have $E(\Q)\cong \Z\oplus \Z/3\Z$, with a generator of the free part given by $P=\left(\frac{23769}{400}, \frac{3529853}{8000}\right)$. One can show that $P$ is strongly $3$-indivisible.

Since the $\mathbb{Q}$-isogeny class of $E$ consists of exactly two curves, by \cite[Theorem 1]{LawsonWutrich} we have
$H^1(H_3,E[3])=\Z/3\Z$. The $3$-torsion field is given by $\Q(z)$, where $z$ is any root of $x^6+3$. Over this field the point
\begin{align*}
Q=\left(\frac{803}{400}z^4 - \frac{416}{400}z^2 + \frac{507}{400}, \frac{89133}{8000}z^4 - \frac{199071}{8000}z^2 - \frac{95323}{8000}\right)\in E(\Q(z))
\end{align*}
is such that $3Q=P$.

A computer search performed with the help of the LMFDB \cite{lmfdb} and of Pari/GP \cite{PARI2} shows that  there are only $20$ elliptic curves with conductor less than $4\times 10^{5}$ satisfying this property for $\ell=3$, none of which has conductor less than $17739$.
\end{exa}

\subsection{Divisibility in the $\ell$-adic torsion tower}

As we have seen in the previous Section, the $\ell$-divisibility of a point can increase when we move along the $\ell$-adic torsion field tower. We would now like to give a bound on the extent of this phenomenon.

Our purpose in this section is to prove Proposition \ref{prop-not-divisible} (essentially an application of Sah's lemma, see \cite[Proposition 2.7 (b)]{MR0229713} and \cite[Lemma A.2]{MR2018998}), which will allow us to give such a bound in terms of the image of the torsion representation.

\begin{lem}
\label{lem-not-divisible}
Let $L$ be a finite Galois extension of $K$ containing $K_{\ell^{n}}$ and let $G:=\gal(L|K)$. Assume that $\ell^{k}H^1(G,E[\ell^n])=0$.
If $\alpha\in E(K)$ is strongly $\ell$-indivisible in $E(K)$, then $\alpha$ is not $\ell^{k+1}$-divisible in $E(L)$.
\begin{proof}
Applying Lemma \ref{lem-divibility-1} with $M=\ell^{k+1}$ we have that the quotient $\frac{E(K)\cap \ell^{k+1}E(L)}{\ell^{k+1} E(K)}$ embeds in $H^1(G,E[\ell^n])$, so it is killed by $\ell^k$. Therefore $\ell^k\left(E(K)\cap \ell^{k+1}E(L)\right)\subseteq\ell^{k+1} E(K)$. Assuming by contradiction that $\alpha\in \ell^{k+1} E(L)$ we get $\ell^{k}\alpha=\ell^{k+1}\beta$ for some $\beta\in E(K)$. But then $T=\ell\beta-\alpha\in E[\ell^{k}](K)$ is such that $\alpha+T\in \ell E(K)$, contradicting our assumption that $\alpha$ is strongly $\ell$-indivisible.
\end{proof}
\end{lem}

\begin{lem}
\label{lemma:ExponentOfH1}
Assume that for some $n_0\geq 1$ we have $(1+\ell^{n_0})\Id \in H_{\ell^n}$ (if $n \leq n_0$ the condition is automatically satisfied). Then the exponent of $H^1(H_{\ell^n},E[\ell^k])$ divides $\ell^{n_0}$ for every $k \leq n$.
\begin{proof}
Let $\lambda=(1+\ell^{n_0})\Id$ and let $\varphi:H_{\ell^n}\to E[\ell^{k}]$ be a cocycle. Using that $\lambda$ is central in $H_{\ell^n}$ and that $\varphi$ is a cocycle, for any $g\in H_{\ell^n}$ we have
\begin{align*}
g\varphi(\lambda)+\varphi(g)=\varphi(g\lambda)=\varphi(\lambda g)= \lambda\varphi(g)+\varphi(\lambda),
\end{align*}
so
\begin{align*}
\ell^{n_0}\varphi(g)=(\lambda-1)\varphi(g)=g\varphi(\lambda)-\varphi(\lambda),
\end{align*}
that is, $\ell^{n_0}\varphi$ is a coboundary. This proves that $\ell^{n_0}H^1(H_{\ell^n},E[\ell^{k}])=0$ as claimed.
\end{proof}
\end{lem}

\begin{lem}
\label{lem-all-mat}
Assume that $E$ does not have complex multiplication and let $n_\ell\geq1$ be a parameter of maximal growth for the $\ell$-adic torsion representation. Then for every $n\geq n_\ell$ and for every $g\in \mat_{2}(\Z_\ell)$ we have that $(\Id + \ell^{n_\ell}g) \bmod \ell^n$ is an element of $H_{\ell^n}$.
\begin{proof}
We prove this by induction. 
For $n=n_\ell$ the statement is trivial, so suppose $(\Id+\ell^{n_\ell}g)\mod \ell^n$ belongs to $H_{\ell^n}$ for some $n> n_\ell$. Since the map $H_{\ell^{n+1}}\to H_{\ell^n}$ is surjective we can lift this element to an element of the form $\Id + \ell^{n_\ell}g + \ell^n g' \in H_{\ell^{n+1}}$, where $g'\in\mat_{2}(\F_\ell)$. Since 
\[
\ker(H_{\ell^{n+1}}\to H_{\ell^n})=\set{\Id + \ell^n h \mid h\in \mat_{2}(\F_\ell)}
\]
we have that $\Id - \ell^n g'$ is in $H_{\ell^{n+1}}$, hence $H_{\ell^{n+1}}$ contains the product
\[
(\Id-\ell^{n}g')(\Id + \ell^{n_\ell} g + \ell^n g') \equiv (\Id + \ell^{n_\ell}g)\pmod{\ell^{n+1}},
\]
where we use the fact that $\ell^{2n}(g')^2 = \ell^{n+n_\ell} g'g = 0 $ since we are working modulo $\ell^{n+1}$.
\end{proof}
\end{lem}

In the special case $g=\Id$, the same result also holds for elliptic curves with complex multiplication:

\begin{lem}\label{lemma:ImageAlwaysContainsScalars}
Let $E$ be an arbitrary elliptic curve and let $n_\ell \geq 1$ be a parameter of maximal growth (in particular, $n_\ell \geq 2$ if $\ell=2$). For every $n \geq n_\ell$ we have $(1+\ell^{n_\ell}) \operatorname{Id}  \in H_{\ell^n}$.
\end{lem}
\begin{proof}
In the light of the previous lemma we may assume that $E$ has complex multiplication, so that the image of the torsion representation is contained in the normaliser of a Cartan subgroup of $\operatorname{GL}_2(\Z_\ell)$.
The equality $\#H_{\ell^{n+1}} = \ell^2 \# H_{\ell^n}$ for $n \geq n_\ell$ is equivalent to the fact that 
\[
\ker \left( H_{\ell^{n+1}} \to H_{\ell^n} \right) = \operatorname{Id} + \ell^n \mathbb{T} \subseteq \{ M \in \operatorname{Mat}_2(\Z/\ell^{n+1}\Z) : M \equiv \operatorname{Id} \pmod{\ell^n} \},
\]
where $\mathbb{T}$ is the tangent space to the image of the Galois representation as introduced in \cite[Definition 9]{2016arXiv161202847L} and further studied in \cite[Definition 18]{MR3690236}. We proceed by induction, the base case $n=n_\ell$ being trivial. By surjectivity of $H_{\ell^{n+1}} \to H_{\ell^n}$ and the inductive hypothesis, we know that $H_{\ell^{n+1}}$ contains an element reducing to $(1+\ell^{n_\ell}) \operatorname{Id}$ modulo $\ell^n$, that is, an element of the form $M_{n+1} := (1 + \ell^{n_\ell})\operatorname{Id} + \ell^n t$. Here $t$ is an element of $\mathbb{T}$: to see this, notice that $M_{n+1}$ is congruent to the identity modulo $\ell^{n_\ell}$, so it cannot lie in the non-trivial coset of the normaliser of a Cartan subgroup (\cite[Theorem 40]{MR3690236}), and therefore belongs to the Cartan subgroup itself. But then $M_{n+1}$ is of the form $\begin{pmatrix}
x & \delta y \\
y & x+\gamma y
\end{pmatrix}$ for appropriate parameters $(\gamma,\delta)$, hence
\[
t=\frac{1}{\ell^n}\begin{pmatrix}
x-1-\ell^{n_\ell} & \delta y \\
y & (x-1-\ell^{n_\ell})+\gamma y
\end{pmatrix} \in \operatorname{Mat}_2(\mathbb{F}_\ell)
\]
belongs to $\mathbb{T}$ by the explicit description given in \cite[Definition 18]{MR3690236}. Using the equality $\ker \left( H_{\ell^{n+1}} \to H_{\ell^n} \right) = \Id + \ell^n \mathbb{T}$ we see that $H_{\ell^{n+1}}$ also contains $\operatorname{Id}-\ell^n t$, so it contains
\[
((1+\ell^{n_\ell})\operatorname{Id} + \ell^n t)(\operatorname{Id}-\ell^n t) \equiv \operatorname{Id}-\ell^{2n} t^2 + \ell^{n_\ell} \operatorname{Id} - \ell^{n+n_\ell} t \equiv (1 + \ell^{n_\ell}) \Id \pmod{\ell^{n+1}}
\]
as claimed.
\end{proof}

\begin{prop}
\label{prop-not-divisible}
Assume that $\alpha$ is strongly $\ell$-indivisible in $E(K)$. Let $n_\ell$ be a parameter of maximal growth for the $\ell$-adic torsion representation. Then for every $n$ the point $\alpha$ is not $\ell^{n_\ell+1}$-divisible in $K_{\ell^n}$; equivalently, $\alpha$ is not $\ell^{n_\ell+1}$-divisible in $K_{\ell^{\infty}}$.
\begin{proof}
By Lemma \ref{lemma:ImageAlwaysContainsScalars} the group $H_{\ell^n}$ contains $(1+\ell^{n_\ell})\Id$, so by Lemma \ref{lemma:ExponentOfH1} the exponent of $H^1(H_{\ell^n}, E[\ell^{n}])$ divides $\ell^{n_\ell}$. We conclude by Lemma \ref{lem-not-divisible}.
\end{proof}
\end{prop}

\subsection{The $\ell$-adic failure is bounded}
In this section we establish some general results that will form the basis of all subsequent arguments (in particular Lemma \ref{lem-order-n-d} and Proposition \ref{prop:GroupTheory}) and use them to show that the $\ell$-adic failure $A_\ell(N)$ can be effectively bounded (Theorem \ref{thm:UpperBoundAl}).

\begin{lem}
\label{lem-order-n-d}
Assume that for some $d\geq 0$ the point $\alpha\in E(K)$ is not $\ell^{d+1}$-divisible over $K_{\ell^\infty}$. Then $V_{\ell^\infty}$ contains a vector of valuation at most $d$.

Similarly, if $\alpha\in E(K)$ is not $\ell^{d+1}$-divisible over $K_{\infty}$ then $W_{\ell^\infty}$ contains a vector of valuation at most $d$.
\begin{proof}
Assume by contradiction that every element of $V_{\ell^\infty}$ has valuation at least $d+1$. Then the image of $V_{\ell^\infty}$ in $E[\ell^{d+1}]=T_\ell(E)/\ell^{d+1}T_\ell(E)$ is zero. As this image is exactly $\gal(K_{\ell^\infty,\ell^{d+1}}\mid K_{\ell^\infty})$, we obtain $K_{\ell^\infty,\ell^{d+1}}=K_{\ell^\infty}$, so $\alpha$ is $\ell^{d+1}$-divisible in $K_{\ell^\infty}$, a contradiction.

The second part can be proved in exactly the same way.
\end{proof}
\end{lem}

The following group-theoretic Proposition will be applied in this section and in Section \ref{sec:UniformBounds}. In all of our applications the group $H$ will be the image of the $\ell$-adic torsion representation associated with some elliptic curve.

\begin{prop}\label{prop:GroupTheory}
Let $\ell$ be a prime number, $d$ be a positive integer, $H$ be a closed subgroup of $\GL_2(\Z_\ell)$, and $A=\Z_\ell[H]$ be the sub-$\Z_\ell$-algebra of $\Mat_2(\Z_\ell)$ generated by the elements of $H$.
Let $V \subseteq \Z_\ell^2$ be an $A$-submodule of $\Z_\ell^2$, and suppose that $V$ contains at least one vector of $\ell$-adic valuation at most $d$.
\begin{enumerate}
\item Suppose that $H$ contains $\{ M \in \operatorname{Mat}_2(\Z_\ell) : M \equiv \operatorname{Id} \pmod{\ell^n}  \}$ for some $n \geq 1$. Then $V$ contains $\ell^{d+n} \Z_\ell^2$.
\item Suppose that the reduction of $H$ modulo $\ell$ acts irreducibly on $\F_\ell^2$. Then $V$ contains $\ell^d  \Z_\ell^2$.
\item Let $C$ be a Cartan subgroup of $\GL_2(\Z_\ell)$ with parameters $(\gamma,\delta)$ and let $N$ be its normaliser. Suppose that $H$ is an open subgroup of $N$ not contained in $C$, and that $H$ contains $\{ M \in C : M \equiv \operatorname{Id} \pmod{\ell^n}  \}$ for some $n \geq 1$. Then $V$ contains $\ell^{3n+d+v_\ell(4\delta)}\Z_\ell^2$.
\end{enumerate}
\end{prop}
\begin{proof}
The assumptions and the conclusions of the Proposition are invariant under changes of basis in $\Z_\ell^2$, so we may assume that $v=\ell^d e_1$ is in $V$, where $e_1=\begin{pmatrix}
1 \\ 0
\end{pmatrix}$. 
\begin{enumerate}
\item It is clear that $A$ contains $\ell^n \Mat_2(\Z_\ell)$, so we have 
\[
V \supseteq A \cdot v \supseteq \ell^n \Mat_2(\Z_\ell) \cdot v = \ell^{n+d} \Mat_2(\Z_\ell) \cdot e_1=\ell^{n+d}\Z_\ell^2.
\]

Let $H_\ell$ denote the reduction of $H$ modulo $\ell$. The condition that $H_\ell$ acts irreducibly on $\mathbb{F}_\ell^2$ implies that there exists $\overline{M} \in \F_\ell[H_\ell]$ such that $\overline{M}e_1 \equiv \begin{pmatrix}
0 \\ 1
\end{pmatrix} \pmod{\ell}$. Fix a lift $M \in A$ of $\overline{M}$, which exists because the natural reduction map $A=\Z_\ell[H] \to \F_\ell[H_\ell]$ is clearly surjective. Then $M v=\ell^d M e_1$ is a vector whose second coordinate has valuation exactly $d$ and whose first coordinate has valuation strictly larger than $d$. It is then immediate to see that $v$ and $Mv$, that are contained in $V$, generate $\ell^d \Z_\ell^2$.
\item It is enough to show that $A$ contains $\ell^{3n+v_\ell(4\delta)}\Mat_2(\Z_\ell)$, and the conclusion follows as in (1) above.

Suppose first that $\gamma=0$, and let $M_0=\begin{pmatrix}
x_0 & -\delta y_0\\
y_0 & -x_0
\end{pmatrix}\in H\setminus C$ and $M_1=\begin{pmatrix}
1+\ell^nx_0 & \delta\ell^ny_0\\
\ell^n y_0 & 1+\ell^nx_0
\end{pmatrix}\in H$. The existence and the form of such matrices follow from the assumptions and from the description of Cartan subgroups and their normaliser given in Definition \ref{def:CartanParameters} and Lemma \ref{lem-Norm}.
Then $A$ contains $M_2=M_1-\Id+\ell^nM_0=2\ell^n\begin{pmatrix}
x_0 & 0 \\ y_0 & 0
\end{pmatrix}$. Let moreover $M_3=\ell^n\begin{pmatrix}
0 & \delta \\ 1 & 0
\end{pmatrix}$, which is in $A$ since it can be written as $\begin{pmatrix}
1 & \ell^n\delta \\ \ell^n & 1
\end{pmatrix}-\Id$, where both matrices are in $H$ by assumption. Then we have
\[
4\ell^{2n}\begin{pmatrix}
x_0^2-\delta y_0^2 & 0\\ 0 & 0
\end{pmatrix} = \left(M_2-2y_0M_3\right)\cdot M_2 \in A
\]
and $x_0^2-\delta y_0^2=-\det M_0\in \Z_\ell^\times$. It follows that $A$ contains $4\ell^{2n}\begin{pmatrix}
1 & 0\\ 0 & 0
\end{pmatrix}$, and since $\Id \in A$ we have that all diagonal matrices of valuation at least $2n+v_\ell(4)$ are in $A$, which therefore also contains
$\displaystyle 
\begin{pmatrix}
0 & 0 \\ \ell^{3n+v_\ell(4)} & 0
\end{pmatrix}= M_3\begin{pmatrix}
\ell^{2n+v_\ell(4)}  & 0 \\ 0 & 0
\end{pmatrix}$ and $\displaystyle \begin{pmatrix}
0 & \ell^{3n+v_\ell(4)}\delta \\ 0 & 0
\end{pmatrix}= M_3\begin{pmatrix}
0 & 0 \\ 0 & \ell^{2n+v_\ell(4)} 
\end{pmatrix}$.
Together with the diagonal matrices found above, these elements clearly generate $\ell^{3n+v_\ell(4\delta)}\Mat_2(\Z_\ell)$, and we are done.

If $\gamma\neq 0$, by Remark \ref{rem:Parameters} we may assume $\gamma =1$ and $\ell=2$. In this case let $M_0=\begin{pmatrix}
x_0+y_0& \delta y_0+x_0+y_0\\
-y_0 & -x_0-y_0
\end{pmatrix}\in H\setminus C$ and $M_1=\Id+\ell^n\begin{pmatrix}
x_0 & \delta y_0\\
y_0 & x_0+y_0
\end{pmatrix}\in H$. Then $A$ contains $M_2=M_1-\Id+\ell^nM_0=\ell^n\begin{pmatrix}
2x_0+y_0 & 2\delta y_0 + x_0 + y_0\\
0 & 0
\end{pmatrix}$. Let moreover $M_3=\ell^n\begin{pmatrix}
-1 & \delta \\
1 & 0
\end{pmatrix}\in A$. Then we have
\[
M_2(\delta M_2 - (2\delta y_0+x_0+y_0) M_3)=
-\ell^{2n}\det(M_0)(1+4\delta)\begin{pmatrix}
1&0\\0&0
\end{pmatrix}\in A,
\]
and using the fact that $\det(M_0) \in \Z_\ell^\times$ (since $M_0 \in H \subseteq \GL_2(\Z_\ell)$) we obtain that $A$ contains all diagonal matrices of valuation at least $2n$. We can then conclude as before.
\end{enumerate}
\end{proof}

\begin{prop}
\label{prop-max-growth-kummer}
Assume that $\alpha$ is strongly $\ell$-indivisible in $E(K)$ and let $n_\ell$ be a parameter of maximal growth for the $\ell$-adic torsion representation.
\begin{enumerate}[(1)]
\item Assume that $E$ does not have complex multiplication.
Then for every $k\geq 1$ we have $E[\ell^k]\subseteq V_{\ell^{k+2n_\ell}}$.
\item Assume that $E$ has complex multiplication by $\order:=\operatorname{End}_{\overline K}(E)$, and that $K$ does not contain the imaginary quadratic field $\order \otimes_{\mathbb{Z}} \mathbb{Q}$.
Let $(\gamma,\delta)$ be parameters for the Cartan subgroup of $\GL_2(\Z_\ell)$ corresponding to $\order$. Then for all $k\geq 1$ we have $E[\ell^k]\subseteq V_{\ell^{k+4n_\ell+v_\ell(4\delta)}}$.
\end{enumerate}
\begin{proof}
By Remark \ref{rem:Vsubgroup}, in order to show part (1) it is enough to prove $\ell^{2n_\ell} T_\ell(E)\subseteq V_{\ell^\infty}$. To see that this holds, notice that by Lemma \ref{lem-order-n-d} and Proposition \ref{prop-not-divisible} there is an element of valuation at most $n_\ell$ in $V_{\ell^\infty}$. Now we just need to apply Proposition \ref{prop:GroupTheory}(1) with $H=H_{\ell^\infty}$, $V=V_{\ell^\infty}$ and $d=n=n_\ell$. Part (2) can be proved in the same way using Proposition \ref{prop:GroupTheory}(3).
\end{proof}
\end{prop}

In \S\ref{sec:CMCounterexample} we will show that a na\"ive analogue of Proposition \ref{prop-max-growth-kummer} does not hold in case $E$ has complex multiplication defined over $K$.

\begin{rem}
\label{rem-parameters}
Write $\alpha=\ell^d\beta+T_h$, where $\beta\in E(K)$ is strongly $\ell$-indivisible and $T_h\in E[\ell^h](K)$ is a point of order $\ell^h$, for some $h,d\geq 0$. Notice that it is always possible to do so: first, let $\beta\in E(K)$ and $d$ be such that $\alpha=\ell^{d}\beta+T$ for some $T\in E(K)$ of order a power of $\ell$, with $d$ maximal. Assume then by contradiction that $\beta$ is not strongly $\ell$-indivisible. This means that there are $\gamma,S\in E(K)$ with $S$ of order a power of $\ell$ such that $\beta =\ell\gamma+S$. But then $\alpha=\ell^d(\ell\gamma+S)+T=\ell^{d+1}\gamma +(\ell^dS+T)$, contradicting the maximality of $d$.
\end{rem}

\begin{rem}\label{rmk:dIsEffective}
Let $\hat{h}$ be the canonical (N\'eron-Tate) height on $E$, as described in \cite[Section VIII.9]{SilvermanEC}. Following \cite{Petsche}, it is possible to bound the divisibility parameters $d$ and $h$ in terms of $\hat h(\alpha)$, the degree of $K$ over $\Q$, the discriminant $\Delta_E$ of $E$ over $K$ and the Szpiro ratio
\begin{align*}
\sigma=\begin{cases}
1 & \text{if $E$ has everywhere good reduction}\\
\frac{\log|N_{K/\Q}(\Delta_{E})|}{\log|N_{K/\Q}(N_{E})|} & \text{otherwise}
\end{cases}
\end{align*}
where $N_E$ denotes the conductor of $E$ over $K$.
In fact, \cite[Theorem 1]{Petsche} gives the bound
\begin{align*}
h\leq v_\ell \left\lfloor c_1[K:\Q]\sigma^2\log\left(c_2[K:\Q]\sigma^2\right)\right\rfloor
\end{align*}
where $c_1=134861$ and $c_2=104613$.

For the parameter $d$ we can reason as follows. For $\alpha=\ell^d \beta+T_h$, by \cite[Theorem 9.3]{SilvermanEC} we have
\begin{align*}
\hat h(\alpha) = \hat h(\ell^d\beta+T_h)=\hat h(\ell^d\beta)=\ell^{2d}\hat h(\beta)
\end{align*}
so we get
$\displaystyle d\leq \frac{1}{2 \log \ell}\log\left(\frac{\hat h(\alpha)}{\hat h(\beta)}\right).$
Now in view of \cite[Theorem 2]{Petsche} for any non-torsion point $\beta\in E(K)$ we have
\begin{align*}
\hat h(\beta)\geq B:=\frac{\log|N_{K/\Q}(\Delta_{E})|}{10^{15}[K:\Q]^3\sigma^6\log^2(c_2[K:\Q]\sigma^2)},
\end{align*}
where again $c_2=104613$. We thus obtain the effective bound
$d\leq \frac{1}{2 \log \ell}\log\left(\frac{\hat h(\alpha)}{B}\right)$.
\end{rem}

\begin{thm}\label{thm:UpperBoundAl}
Let $\ell$ be a prime and assume that $\operatorname{End}_K(E)=\Z$ (i.e.~either $E$ does not have CM, or it has CM but the complex multiplication is not defined over $K$). There is an effectively computable constant $a_\ell$, depending only on $\alpha$ and on the $\ell$-adic torsion representation associated to $E$, such that $A_\ell(N)$ divides $\ell^{a_\ell}$ for all positive integers $N$.

Moreover, $a_\ell$ is zero for every odd prime $\ell$ such that $\alpha$ is $\ell$-indivisible and for which the $\ell$-adic torsion representation associated with $E$ is maximal (see Definition \ref{def:MaximalRepresentation}). For the finitely many remaining primes $\ell$ we can take $a_\ell$ as follows: let $n_\ell$ be a parameter of maximal growth for the $\ell$-adic torsion representation and let $d$ be as in Remark \ref{rem-parameters}. If $E$ has CM over $\overline K$, let $(\gamma,\delta)$ be parameters for the Cartan subgroup of $\GL_2(\Z_\ell)$ corresponding to $\operatorname{End}_{\overline K}(E)$. Then:
\begin{itemize}
\item $a_\ell=4n_\ell+2d$ if $E$ does not have CM over $\overline K$;
\item $a_\ell=8n_\ell+2v_\ell(4\delta)+2d$ if $E$ has CM over $\overline K$.
\end{itemize}
\begin{proof}
Let $\alpha=\ell^d\beta+T_h$ as described above. Notice that if $\alpha$ is strongly $\ell$-indivisible we have $d=0$, and the conclusion follows from Proposition \ref{prop-max-growth-kummer}. If the $\ell$-adic torsion representation is maximal, the fact that $a_\ell$ is zero in the cases stated follows from \cite[Theorem 5.2 and Theorem 5.8]{JonesRouse}.

We now study the $\ell$-adic failure $A_\ell(N)$ in the general case. Let $n=v_\ell(N)$ and notice that the claim is trivial for $n\leq d$, so we may assume $n> d$. Since
\begin{align*}
[K_{\ell^{n+h}}(\ell^{-n}\alpha):K_{\ell^{n+h}}]=[K_{\ell^{n}}(\ell^{-n}\alpha) K_{\ell^{n+h}}:K_{\ell^n}K_{\ell^{n+h}}]\quad \text{divides}\quad [K_{\ell^{n}}(\ell^{-n}\alpha):K_{\ell^n}]
\end{align*}
 we have that
$\displaystyle \frac{\ell^{2n}}{[K_{\ell^{n}}(\ell^{-n}\alpha):K_{\ell^n}]}$ divides $\displaystyle  \frac{\ell^{2n}}{[K_{\ell^{n+h}}(\ell^{-n}\alpha):K_{\ell^{n+h}}]}
$,
and since we have 
\begin{align*}
K_{\ell^{n+h}}(\ell^{-n}\alpha)=K_{\ell^{n+h}}(\ell^{-(n-d)}\beta)
\end{align*}
we get
\begin{align*}
\frac{\ell^{2n}}{[K_{\ell^{n+h}}(\ell^{-n}\alpha):K_{\ell^{n+h}}]}=\ell^{2d}\frac{\ell^{2(n-d)}}{[K_{\ell^{n+h}}(\ell^{-(n-d)}\beta):K_{\ell^{n+h}}]}
\end{align*}
so in view of Remark \ref{rem-NM} we are reduced to proving the statement for $\beta$ instead of $\alpha$. Since $\beta$ is strongly $\ell$-indivisible, we can conclude as stated at the beginning of the proof.

The fact that $a_\ell$ is effective follows from the fact that one can effectively compute a parameter of maximal growth for the $\ell$-adic torsion representation (Remark \ref{rmk:nlIsEffective}), an upper bound for the value of $d$ (Remark \ref{rmk:dIsEffective}), and the endomorphism ring $\operatorname{End}_{\overline K}(E)$ (\cite{Achter}, \cite{CMSV}, \cite{LombardoEndo}).
\end{proof}
\end{thm}

\section{The adelic failure}
\label{sec:AdelicFailure}

In this section we study the adelic failure $B_\ell(N)$, that is, the degree of the intersection $K_{\ell^{n},\ell^{n}}\cap K_N$ over $K_{\ell^{n}}$. Notice that this intersection is a finite Galois extension of $K_{\ell^n}$.

\subsection{Intersection of torsion fields in the non-CM case}\label{sec:IntersectionTorsionFields}

We first aim to establish certain properties of the intersections of different torsion fields of $E$, assuming for this subsection that $E$ does not have complex multiplication over $\overline{K}$. Our main tool (Theorem \ref{th-intersection}) is a refinement of \cite[Theorem 3.3.1]{Campagna}, and will appear in an upcoming paper of F.~Campagna and P.~Stevenhagen. The proof of the stronger version we need requires only minor changes with respect to that of \cite[Theorem 3.3.1]{Campagna}, and can be easily derived from it using the following well-known lemmas (see \cite{Serre} and \cite{MR1484415}).

\begin{lem}
\label{lemma-solvable}
Let $p$ be a prime and let $H$ be a subgroup of $\GL_2(\F_p)$. Let $S$ be a non-abelian simple group that occurs in $H$. Then $S$ is isomorphic either to $A_5$ or to $\PSL_2(\F_p)$; the latter case is only possible if $H$ contains $\SL_2(\F_p)$.
\end{lem}

\begin{lem}[Serre]\label{lemma:SerreLifting}
Let $\ell \geq 5$ be a prime and let $G \subseteq \SL_2(\Z/\ell^k\Z)$ be a subgroup. Let $\pi : \SL_2(\Z/\ell^k\Z) \to \SL_2(\Z/\ell\Z)$ be the reduction homomorphism and suppose that $\pi(G) = \SL_2(\Z/\ell\Z)$: then $G=\SL_2(\Z/\ell^k\Z)$.
\end{lem}

\begin{thm}
\label{th-intersection}
Assume that $E$ does not have complex multiplication. Let $S$ be the set consisting of the primes $\ell$ satisfying one or more of the following three conditions:
\begin{enumerate}[(i)]
\item $\ell\mid 30\disc(K\mid \mathbb{Q})$;
\item $E$ has bad reduction at some prime of $K$ above $\ell$;
\item the modulo $\ell$ torsion representation is not surjective.
\end{enumerate}
For every $\ell\not \in S$ we have $K_{\ell^n}\cap K_M=K$ for all $M,n\geq 1$ with $\ell\nmid M$. 
\end{thm}

\begin{rem}
The finite set $S$ appearing in Theorem \ref{th-intersection} can be computed explicitly. 
In fact, it is well known that one can compute the discriminant of $K$ and the set of primes of bad reduction of $E$. An algorithm to compute the set of primes for which the $\bmod\, \ell$ representation is not surjective is described in \cite{Zywina}. 
\end{rem}

As a corollary, we give a slightly more precise version of \cite[\S3.4, Lemma 6]{MR1484415}.

\begin{cor}
\label{cor-direct-prod}
Assume that $E$ does not have complex multiplication and let $S$ be as in Theorem \ref{th-intersection}. Let $M$ be a positive integer and write $M=M_1 M_2$, where
\begin{align*}
M_1&=\prod_{p\not \in S}p^{e_p}&& p\text{ prime, }e_p\geq 0,\\
M_2&=\prod_{q\in S}q^{e_q} && q\text{ prime, }e_q\geq 0.
\end{align*}
Then we have
\begin{align*}
\gal(K_M\mid K)\cong\GL_2\left( \Z/M_1\Z\right)\times \gal\left(K_{M_2}\mid K\right).
\end{align*}
\begin{proof}
By Theorem \ref{th-intersection} we have that, for any $p\not \in S$ and any $e\geq 0$, the field $K_{p^e}$ is linearly disjoint over $K$ from $K_{M_2}$ and from $K_{q^f}$ for every $q \neq p$ and every $f\geq 1$. Moreover we have
\begin{align*}
\GL_2\left( \Z/M_1\Z\right)\cong \prod_{p\not\in S}\GL_2(\Z/p^{e_p}\Z)\cong  \prod_{p\not\in S}\gal(K_{p^{e_p}}\mid K),
\end{align*}
and the Corollary follows by standard Galois theory.
\end{proof}
\end{cor}

\begin{rem}
\label{rem-S-stable}
Let $\tilde K$ be the compositum of the fields $K_p$ for all $p\in S$, where $S$ is as in Theorem \ref{th-intersection}. In the following section it will be important to notice that $S$ is stable under base change to $\tilde K$. More precisely, let $\tilde S$ be the set of all primes $\ell$ that satisfy one of the following:
\begin{enumerate}[(i')]
\item $\ell\mid 30\disc(\tilde K\mid \mathbb{Q})$;
\item $E$ has bad reduction at some prime of $\tilde K$ above $\ell$;
\item the modulo $\ell$ torsion representation attached to $E/\tilde{K}$ is not surjective.
\end{enumerate}
Then $\tilde S=S$.

Indeed, the inclusion $\tilde S\supseteq S$ is easy to see: clearly conditions (i) and (iii) imply (i') and (iii') respectively, so we only need to discuss (ii). Let $\mathfrak{p}$ be a prime of $K$ (of characteristic $\ell$) at which $E$ has bad reduction, and let $\mathfrak{q}$ be a prime of $\tilde{K}$ lying over $\mathfrak{p}$. We need to show that $\ell \in \tilde{S}$. If $E$ has bad reduction at $\mathfrak{q}$ we have $\ell \in \tilde{S}$ by (ii'), while if $E$ has good reduction at $\mathfrak{q}$ then $\mathfrak{p}$ ramifies in $\tilde{K}$ by \cite[Proposition VII.5.4 (a)]{SilvermanEC}, so we have $\ell \mid \operatorname{disc}(\tilde{K} \mid \Q)$ and $\ell$ is in $\tilde{S}$ by (i').

Conversely, let $\ell \in \tilde{S}$. If (ii') holds, then clearly also (ii) holds, and $\ell$ is in $S$. Suppose that (i') holds. If $\ell$ divides $30$, then it is in $S$ by (1). Otherwise $\ell$ divides $\disc(\tilde K\mid \Q)$, which by \cite[III.§4, Proposition 8]{Serre-LocalFields} is equal to $\disc(K \mid \Q)^{[\tilde{K}:K]} N_{K/\Q} \disc(\tilde{K} \mid K)$; if $\ell$ divides $\disc(K \mid \Q)$, then it is in $S$ by (1), while if it divides $\disc(\tilde{K} \mid K)$ then we have $\ell\in S$ by \cite[Proposition VIII.1.5(b)]{SilvermanEC}.
We may therefore assume that (i') and (ii') do not hold. Since $\ell$ is in $\tilde{S}$, (iii') must hold, that is, the modulo-$\ell$ torsion representation attached to $E/\tilde{K}$ is not surjective. We claim that the same is true for $E/K$. Indeed, if $\ell$ is in $S$ this is true by definition, while if $\ell \not \in S$ the previous corollary shows that $K_{\ell}$ is linearly disjoint from $\tilde{K}$, so the images of the modulo-$\ell$ representations over $K$ and over $\tilde{K}$ coincide.
\end{rem}

\subsection{The adelic failure is bounded}

We now go back to the general case of $E$ possibly admitting complex multiplication.

Fix an integer $N>1$ and a prime number $\ell$ dividing $N$. Write $N=\ell^nR$ with $\ell\nmid R$ and recall that the adelic failure $B_\ell(N)$ is defined to be the degree $[K_{\ell^{n},\ell^{n}}\cap K_N:K_{\ell^{n}}]$. In this section we study this failure for $N=\ell^nR$, starting with a simple Lemma in Galois theory.

\begin{lem}
\label{lem-gal}
Let $L_1$, $L_2$ and $L_3$ be Galois extensions of $K$ with $L_1\subseteq L_2$. Then the compositum $L_1(L_2\cap L_3)$ is equal to the intersection $L_2\cap (L_1L_3)$.
\begin{proof}
For $i=1,2,3$ let $G_i:=\gal(\overline{K}\mid L_i)$. The claim is equivalent to $G_1\cap(G_2\cdot G_3)=G_2\cdot(G_1\cap G_3)$, where the inclusion ``$\supseteq$'' is obvious. Let then $g\in G_1\cap(G_2\cdot G_3)$, so that there are $g_1\in G_1$, $g_2\in G_2$ and $g_3\in G_3$ such that $g=g_1=g_2g_3$. But then $g_2^{-1}g_1=g_3\in G_3$ and, since $G_2\subseteq G_1$, also $g_2^{-1}g_1\in G_1$, so that $g=g_2(g_2^{-1}g_1)\in G_2\cdot(G_1\cap G_3)$.
\end{proof}
\end{lem}

We now establish some properties of certain subfields of $K_{\ell^nR,\ell^n}$.

\begin{lem}
\label{nice-lemma}
Setting
\begin{align*}
L:=K_{\ell^n,\ell^n}\cap K_{N}, && F:= L\cap K_R=K_{\ell^n,\ell^n}\cap K_R, && T:= F\cap K_{\ell^n}=K_{\ell^n}\cap K_R
\end{align*}
we have:
\begin{enumerate}
\item[(a)] The compositum $F K_{\ell^n}$ is $L$.
\item[(b)] $\gal(F\mid T)\cong \gal(L\mid K_{\ell^n})$; in particular, $\gal(F\mid T)$ is an abelian $\ell$-group.
\item[(c)] $F$ is the intersection of the maximal abelian extension of $T$ contained in $K_{\ell^n,\ell^n}$ and the maximal abelian extension of $T$ contained in $K_R$.
\end{enumerate}
\begin{proof}
(a) By Lemma \ref{lem-gal} we have $FK_{\ell^n}=K_{\ell^n}(K_{\ell^n,\ell^n}\cap K_R)=K_{\ell^n,\ell^n}\cap K_{\ell^nR}=L$. (b) Follows from (a) and standard Galois theory. For (c), notice that $F$ is abelian over $T$ by (b), so it must be contained in the maximal abelian extension of $T$ contained in $K_{\ell^n,\ell^n}$ and in the maximal abelian extension of $T$ contained in $K_R$. On the other hand, $F$ cannot be smaller than the intersection of these abelian extensions, because by definition it is the intersection of $K_{\ell^n,\ell^n}$ and $K_R$.
\end{proof}
\end{lem}

\begin{figure}[t]
\begin{equation*}
\xymatrix{
& K_{\ell^nR,\ell^n} \ar@{-}[dl]\ar@{-}[dr]\\
K_{\ell^n,\ell^n}\ar@{-}[dd]\ar@{-}[dr] & & K_{\ell^nR}\ar@{-}[dl]\ar@{-}[dd]\\
& L:=K_{\ell^n,\ell^n}\cap K_{\ell^nR} \ar@{-}[dl]\ar@{-}[dd]\\
K_{\ell^n}\ar@{-}[dddr] & & K_R\ar@{-}[dl]\ar@{-}[dddl]\\
& F:=K_{\ell^n,\ell^n}\cap K_R\ar@{-}[dd]\\
\\
&T:=K_{\ell^n}\cap K_R\ar@{-}[d]\\
&K
}
\end{equation*}
\caption{The situation described in Lemma \ref{nice-lemma} and Proposition \ref{prop-1}.}
\end{figure}
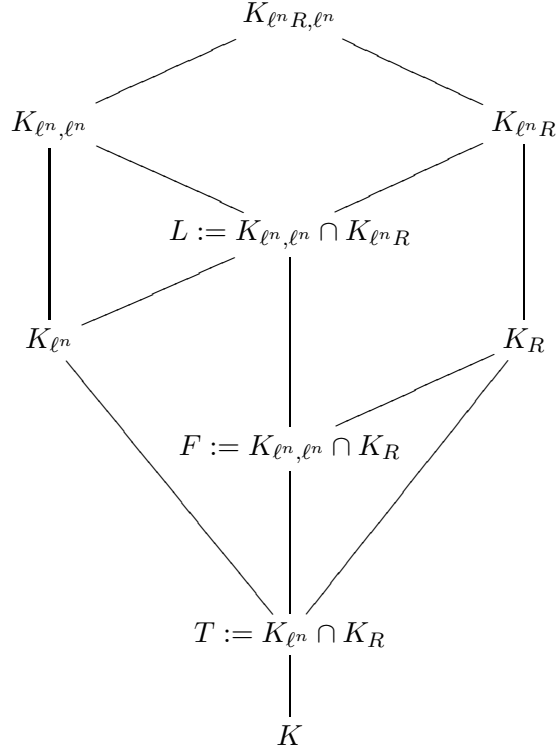

\begin{prop}
\label{prop-1}
The adelic failure $B_\ell(N)$ is equal to $[F:T]$, where $F=K_{\ell^n,\ell^n}\cap K_{R}$ and $T=K_{\ell^n}\cap K_R$.
\begin{proof}
Let as above $L=K_{\ell^n,\ell^n}\cap K_{\ell^nR}$. 
We have $\gal(K_{\ell^n,\ell^n}\,|\,L)\cong \gal(K_{\ell^nR,\ell^n}\,|\,K_{\ell^nR})$, so we get
\begin{align*}
[K_{\ell^n,\ell^n}:K_{\ell^n}]=[K_{\ell^n,\ell^n}:L][L:K_{\ell^n}]=[K_{\ell^nR,\ell^n}:K_{\ell^nR}][L:K_{\ell^n}]
\end{align*}
and we conclude by Lemma \ref{nice-lemma}(b).
\end{proof}
\end{prop}

In what follows we will need to work over a certain extension $\tilde K$ of $K$; this extension will depend on the prime $\ell$. More precisely, we give the following definition.

\begin{defi}
\label{def-tildeK}
Let $\tilde K$ be the finite extension of $K$ defined as follows:

\begin{itemize}
\item If $E$ has complex multiplication, we take $\tilde K$ to be the compositum of $K$ with the CM field of $E$. This is an at most quadratic extension of $K$.
Notice that in this case by \cite[Lemma 2.2]{2018arXiv180902584L} we have $\tilde K_n=K_n$ for every $n\geq 3$.
\item If $E$ does not have CM and $\ell$ is not one of the primes in the set $S$ of Theorem \ref{th-intersection}, we just let $\tilde K = K$. Notice that this happens for all but finitely many primes $\ell$.
\item If $E$ does not have CM and $\ell$ is one of the primes in the set $S$ of Theorem \ref{th-intersection}, we let $\tilde K$ be the compositum of all the $K_p$ for $p\in S$. Notice that in this case $\tilde{K}_\ell=\tilde{K}$.
\end{itemize}

We shall use the notation $\tilde{K}_M$ (respectively $\tilde{K}_{M,N}$) for the torsion (resp.~Kummer) extensions of $\tilde{K}$. We shall also write
\begin{align*}
\tilde H_{\ell^n}&:=\im\left(\tau_{\ell^n}:\gal(\Kbar\mid\tilde K)\to \Aut(E[\ell^n]) \right)\cong \gal\left(\tilde K_{\ell^n}\mid\tilde K\right),\\
\tilde V_{\ell^n}&:=\im\left(\kappa_{\ell^n}:\gal(\Kbar\mid\tilde K_{\ell^n})\to E[\ell^n]\right)\cong \gal\left(\tilde K_{\ell^n,\ell^n}\mid\tilde K_{\ell^n}\right)
\end{align*}
for the images of the  $\ell^n$-torsion representation and of the $(\ell^n,\ell^n)$-Kummer map attached to $E/\tilde{K}$.
Finally, we let $\tilde n_\ell$ be the minimal parameter of maximal growth for the $\ell$-adic torsion representation over $\tilde K$. Notice that, thanks to Lemma \ref{lem-param-tilde-bound}, we have $\tilde n_\ell \leq n_\ell+v_\ell([\tilde K:K])$.
\end{defi}

\begin{prop}
\label{prop-F-abelian}
The extension $F':=\tilde K_{\ell^n,\ell^n}\cap \tilde K_R$ is abelian over $\tilde K$.
\begin{proof}
This is well known if $E$ has complex multiplication because then $\tilde{K}_R$ is itself abelian over $\tilde K$, see for example \cite[Theorem II.2.3]{MR1312368}. In case $E$ does not have complex multiplication and $\ell$ is not in the set $S$ of Theorem \ref{th-intersection}, this follows easily by considering the diagram
\begin{align*}
\xymatrix{
& \tilde K_{\ell^n,\ell^n} \ar@{-}[d]\\
& \tilde K_{\ell^n}F'\ar@{-}[dl]\ar@{-}[dr]\\
\tilde  K_{\ell^n} \ar@{-}[dr]& & F'\ar@{-}[dl]\\
&\tilde K
}
\end{align*}
In fact, since $\tilde K_{\ell^n}\cap F'=\tilde K$ by Theorem \ref{th-intersection} (notice that in this case $\tilde{K}=K$), we have that $\gal(F'\mid \tilde K)\cong \gal(\tilde K_{\ell^n} F' \mid \tilde K_{\ell^n})$ is a quotient of $\tilde V_{\ell^n}$, hence abelian. Thus we can assume that $E$ does not have CM and that $\ell$ is in the set $S$ of Theorem \ref{th-intersection}.

Notice that $F'$ is a Galois extension of $\tilde K$ with degree a power of $\ell$, since the same is true for $\tilde K_{\ell^n,\ell^n} \mid \tilde{K}$ and $F' \subseteq \tilde K_{\ell^n,\ell^n}$. Letting $r$ denote the radical of $R$, the degree of $[F' : F' \cap \tilde{K}_r]$, which is still a power of $\ell$, divides $[\tilde{K}_R : \tilde{K}_r]$, which is a product of primes dividing $R$. So since $\ell \nmid R$ we obtain $[F' : F' \cap \tilde{K}_r]=1$, that is $\tilde{K}_{\ell^n, \ell^n} \cap \tilde{K}_R =\tilde{K}_{\ell^n, \ell^n} \cap \tilde{K}_r$, and we may assume that $R$ is squarefree. Write now $R=R_1R_2$, where $R_1$ is the product of the prime factors of $R$ that are \textit{not} in $S$ and $R_2$ is the product of the prime factors of $R$ that belong to $S$. By definition of $\tilde K$ we have $\tilde{K}_{R} = \tilde{K}_{R_1}$, so we may further assume that no prime $p\in S$ divides $R$. By Corollary \ref{cor-direct-prod} we then have $\gal(\tilde K_{R}\,|\,\tilde K)\cong \GL_2(\Z/R\Z)$.

Since $F'\subseteq \tilde K_{R}$, there must be a normal subgroup $D=\gal(\tilde{K}_R\mid F')\trianglelefteq \GL_2(\Z/R\Z)$ of index a power of $\ell$. In order to conclude we just need to show that $D$ contains $\SL_2(\Z/R\Z)$, for then $\gal( F'\mid \tilde{K})\cong \GL_2(\Z/R\Z)/D$ is abelian.

Since $\SL_2(\Z/R\Z)\cong\prod_{p\mid R}\SL_2(\F_p)$, we can consider the intersection $D_p:=D\cap \SL_2(\F_p)$, which is a normal subgroup of $\SL_2(\F_p)$. Here we identify $\SL_2(\F_p)$ with the corresponding direct factor of $\SL_2(\Z/R\Z)$. The quotient $\SL_2(\F_p)/D_p$ cannot have order a power of $\ell$ unless it is trivial (recall that in our case $p\geq 5$), so we deduce that $D\supseteq \SL_2(\F_p)$. As this is true for every $p\mid R$, we have $D\supseteq \SL_2(\Z/R\Z)$, and we are done.
\end{proof}
\end{prop}

In what follows, whenever $A$ is an abelian group and $Q$ is a group acting on $A$, we denote by $[A,Q]$ the subgroup of $A$ generated by elements of the form $gv-v$ for $v\in A$ and $g\in Q$. For example, we will consider the case $A=\tilde V_{\ell^n}$ and $Q=\tilde H_{\ell^n}$.

\begin{lem}
\label{lemma-ab}
Let 
\begin{align*}
1\to A\to G\to Q \to 1
\end{align*}
be a short exact sequence of groups, with $A$ abelian, so that $Q$ acts naturally on $A$. Let $G^{\ab}$ and $Q^{\ab}$ be the maximal abelian quotients of $G$ and $Q$ respectively. Then $A/[A,Q]$ surjects onto $\ker(G^{\ab}\to Q^{\ab})$.
\begin{proof}
We have an injective map of short exact sequences
\begin{center}
\begin{tikzcd}
1\arrow[r] & A\cap G'\arrow[r]\arrow[d, hook] & G'\arrow[r]\arrow[d, hook] & Q' \arrow[r]\arrow[d, hook] & 1\\
1\arrow[r] & A \arrow[r]& G \arrow[r]& Q\arrow[r]&1
\end{tikzcd}
\end{center}
from which we get the exact sequence
\begin{align*}
1\to \frac{A}{A\cap G'}\to G^{\ab}\to Q^{\ab}\to 1
\end{align*}
and since $[A,Q]\subseteq A\cap G'$ we get that $A/[A,Q]$ surjects onto $A/A\cap G'=\ker(G^{\ab}\to Q^{\ab})$.
\end{proof}
\end{lem}

\begin{prop}
\label{bound-F-over-K}
The adelic failure $B_\ell(N)$ divides $\displaystyle [\tilde{K}:K]\cdot \#\frac{\tilde V_{\ell^n}}{[\tilde V_{\ell^n},\tilde H_{\ell^n}]}$.
\begin{proof}
Let $\campo_1$ and $\campo_2$ be the maximal abelian extensions of $\tilde K$ contained in $\tilde K_{\ell^n}$ and $\tilde K_{\ell^n,\ell^n}$ respectively. Then we have $\gal(\campo_1\,|\,\tilde K)=\tilde H_{\ell^n}^{\ab}$ and $\gal(\campo_2\,|\,\tilde K)=\tilde G_{\ell^n}^{\ab}$, where $\tilde G_{\ell^n}=\gal(\tilde K_{\ell^n,\ell^n}\,|\,\tilde K)$.
Notice that $[\campo_2:\campo_1]=\#W$, where $W=\ker(\tilde G_{\ell^n}^{\ab}\to \tilde H_{\ell^n}^{\ab})$ is a quotient of $\tilde V_{\ell^n}/[\tilde V_{\ell^n},\tilde H_{\ell^n}]$ by Lemma \ref{lemma-ab}.
Let moreover $F':=\tilde K_{\ell^n,\ell^n}\cap \tilde K_R$ and $T':= \tilde K_{\ell^n}\cap \tilde K_R$. By Proposition \ref{prop-F-abelian} we have $F'\subseteq \campo_2$ and clearly also $T'\subseteq \campo_1$ (indeed $T'$ is abelian over $\tilde{K}$ since it is a sub-extension of $F'$).
Consider the compositum $ \campo_1F'$ inside $\campo_2$.
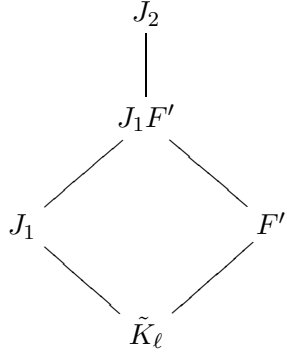
\begin{align*}
\xymatrix{
& \campo_2 \ar@{-}[d]\\
& \campo_1F'\ar@{-}[dl]\ar@{-}[dr]\\
\campo_1 \ar@{-}[dr]& & F'\ar@{-}[dl]\\
&\tilde K_\ell 
}
\end{align*}
It is easy to check that $F'\cap \campo_1=T'$, so we have that $[F':T']=[\campo_1F':\campo_1]$ divides $[\campo_2:\campo_1]$, which in turn divides $\tilde V_{\ell^n}/[\tilde V_{\ell^n},\tilde H_{\ell^n}]$.

Now applying Proposition \ref{prop-1} with $\tilde K$ in place of $K$ we get that
\begin{align*}
\frac{[\tilde K_{\ell^n,\ell^n}:\tilde K_{\ell^n}]}{[\tilde K_{\ell^nR,\ell^n}:\tilde K_{\ell^nR}]}\qquad \text{divides}\qquad [F':T'],
\end{align*}
and using that $[\tilde K_{\ell^nR,\ell^n}:\tilde K_{\ell^nR}]$ divides $[K_{\ell^nR,\ell^n}:K_{\ell^nR}]$ it is easy to see that
\begin{align*}
\frac{[K_{\ell^n,\ell^n}:K_{\ell^n}]}{[K_{\ell^nR,\ell^n}:K_{\ell^nR}]}\qquad \text{divides}\qquad [\tilde K:K]\cdot\frac{[\tilde K_{\ell^n,\ell^n}:\tilde K_{\ell^n}]}{[\tilde K_{\ell^nR,\ell^n}:\tilde K_{\ell^nR}]}.
\end{align*}
We conclude that
\begin{align*}
B_\ell(N)=\frac{[K_{\ell^n,\ell^n}:K_{\ell^n}]}{[K_{\ell^nR,\ell^n}:K_{\ell^nR}]} \qquad \text{divides} \qquad [\tilde{K}:K]\cdot \#\frac{\tilde V_{\ell^n}}{[\tilde V_{\ell^n},\tilde H_{\ell^n}]}.
\end{align*}
\end{proof}
\end{prop}

So we are left with giving an upper bound on the ratio $\#\tilde V_{\ell^n}/\#[\tilde V_{\ell^n},\tilde H_{\ell^n}]$: this is achieved in the following Proposition.

\begin{prop}
\label{prop-quotient-bounded}
For every $n$, the order of $\tilde V_{\ell^n}/[\tilde V_{\ell^n},\tilde H_{\ell^n}]$ divides $\ell^{2\tilde n_\ell}$, where $\tilde n_\ell$ is the minimal parameter of maximal growth for the $\ell$-adic torsion representation of $E/\tilde K$.
\end{prop}
\begin{proof}
By Lemma \ref{lemma:ImageAlwaysContainsScalars}, the group $\tilde H_{\ell^n}$ contains $(1+\ell^{\tilde n_\ell}) \operatorname{Id}$. This implies that for every $v \in \tilde V_{\ell^n}$ the group $[\tilde V_{\ell^n}, \tilde H_{\ell^n}]$ contains 
\[
\left[v, (1+\ell^{\tilde n_\ell}) \operatorname{Id}\right] = (1+\ell^{\tilde{n}_\ell}) \operatorname{Id} \cdot v - v = \ell^{\tilde n_\ell} v,
\]
that is, $[\tilde V_{\ell^n}, \tilde H_{\ell^n}]$ contains $\ell^{\tilde n_\ell} \tilde V_{\ell^n}$.
The claim now follows from the fact that $\tilde{V}_{\ell^n}$ is generated over $\Z/\ell^n\Z$ by at most two elements.
\end{proof}

\begin{lem}
\label{lem-adelic-good-ell}
Assume that $\ell\geq 5$ is unramified in $K\mid \Q$ and that the image of the $\bmod\,\ell$ torsion representation is $\GL_2(\F_\ell)$ (so in particular $E$ does not have CM over $\Kbar$). Assume moreover that $\alpha$ is $\ell$-indivisible. Then $V_{\ell^n}=[V_{\ell^n},H_{\ell^n}]$.
\end{lem}
\begin{proof}
Since $H_{\ell^\infty}'$ is a closed subgroup of $\SL_2(\Z_\ell)$ whose reduction modulo $\ell$ contains $H_\ell'=\GL_2(\F_\ell)'=\SL_2(\F_\ell)$, by Lemma \ref{lemma:SerreLifting} the group $H_{\ell^\infty}$ contains $\SL_2(\Z_\ell)$. The assumption that $\ell$ is unramified in $K$ implies that $\det(H_{\ell^\infty})=\Z_\ell^\times$, which together with the inclusion $\SL_2(\Z_\ell) \subseteq H_{\ell^\infty}$ implies $H_{\ell^\infty}=\GL_2(\Z_\ell)$, and
in particular $H_{\ell^n} = \GL_2(\Z/\ell^n\Z)$. By \cite[Theorem 5.2]{JonesRouse} we have $V_{\ell^n} = (\Z/\ell^n\Z)^2$, so it is enough to consider
\begin{align*}
\begin{array}{ccc}
g_1:=\left(\begin{array}{cc}
1 & 1 \\
0 & 1
\end{array}
\right)\in H_{\ell^n}, &
g_2:=\left(\begin{array}{cc}
1 & 0 \\
1 & 1
\end{array}
\right) \in H_{\ell^n},&
v:=\left(\begin{array}{c}
1\\
1
\end{array}\right)\in V_{\ell^n}
\end{array}
\end{align*}
to conclude that 
\begin{align*}
\begin{array}{ccc}
\left(\begin{array}{c}
1\\
0
\end{array}\right)=g_1v-v\in [V_{\ell^n},H_{\ell^n}] &
\text{and} &
\left(\begin{array}{c}
0\\
1
\end{array}\right)=g_2v-v\in [V_{\ell^n},H_{\ell^n}],
\end{array}
\end{align*}
so that $V_{\ell^n}\subseteq[V_{\ell^n},H_{\ell^n}]$.
\end{proof}

\begin{lem}\label{lemma:CMVanishing}
Let $E/K$ be an elliptic curve such that $\End_{\overline{K}}(E)$ is an order $\order$ in the imaginary quadratic field $\Q(\sqrt{-d})$. Let $\ell \geq 3$ be a prime unramified both in $K$ and in $\mathbb{Q}(\sqrt{-d})$, and suppose that $E$ has good reduction at all places of $K$ of characteristic $\ell$. Then $V_{\ell^n}=[V_{\ell^n},H_{\ell^n}]$ and $\tilde{V}_{\ell^n}=[\tilde{V}_{\ell^n},\tilde{H}_{\ell^n}]$.
\end{lem}
\begin{proof}
By \cite[Theorem 1.5]{MR3766118}, the image of the $\ell$-adic representations attached to both $E/K$ and $E/\tilde{K}$ contains $(\order \otimes \Z_\ell)^\times$, hence in particular it contains an operator that acts as multiplication by $2$ on $E[\ell^n]$ for every $n$. Let $\lambda$ be such an operator: then $[V_{\ell^n},H_{\ell^n}]$ contains $[V_{\ell^n}, \lambda] = \{\lambda v - v \bigm\vert v \in V_{\ell^n} \} = V_{\ell^n}$ as claimed. The case of $\tilde{V}_{\ell^n}$ is similar.
\end{proof}

\begin{thm}\label{thm:UpperBoundBl}
Let $\ell$ be a prime. There is a constant $b_\ell$, depending only on the $p$-adic torsion representations associated with $E$ for all the primes $p$, such that $B_\ell(N)$ divides $\ell^{b_\ell}$ for all positive integers $N$.
Moreover,
\begin{itemize}
\item Suppose that $E$ does not have complex multiplication over $\overline{\mathbb{Q}}$. Then $b_\ell$ is zero whenever the following conditions all hold: $\alpha$ is $\ell$-indivisible, $\ell >  5$ is unramified in $K\mid \Q$, the $\bmod \, \ell$ torsion representation is surjective, and $E$ has good reduction at all places of $K$ of characteristic $\ell$.
\item Suppose that $\End_{\overline{K}}(E)$ is an order in the imaginary quadratic field $\mathbb{Q}(\sqrt{-d})$. Then $b_\ell$ is zero whenever the following conditions all hold: $\ell \geq 3$ is a prime unramified both in $K$ and in $\mathbb{Q}(\sqrt{-d})$, and $E$ has good reduction at all places of $K$ of characteristic $\ell$.
\end{itemize}

Both in the CM and non-CM cases, for the finitely many remaining primes $\ell$ we can take $b_\ell=2n_\ell+3v_\ell\left([\tilde K:K]\right)$, where $\tilde K$ is as in Definition \ref{def-tildeK} and $n_\ell$ is a parameter of maximal growth for the $\ell$-adic torsion part.

\begin{proof}
Let $n$ be the $\ell$-adic valuation of $N$. By Proposition \ref{bound-F-over-K}, the adelic failure $B_\ell(N)$ divides $\displaystyle [\tilde{K}:K] \cdot \# \frac{\tilde{V}_{\ell^n}}{[\tilde{V}_{\ell^n}, \tilde{H}_{\ell^n}]}$.

\begin{itemize}
\item Suppose that $E$ does not have CM over $\overline{\Q}$, that $\alpha$ is $\ell$-indivisible, that $\ell > 5$ is unramified in $K\mid \Q$, that the $\bmod\,\ell$ torsion representation is surjective, and that $E$ has good reduction at all places of $K$ of characteristic $\ell$. Under these assumptions, the prime $\ell$ does not belong to the set $S$ of Theorem \ref{th-intersection}, so we have $\tilde{K}=K$ and $\displaystyle [\tilde{K}:K] \cdot \# \frac{\tilde{V}_{\ell^n}}{[\tilde{V}_{\ell^n}, \tilde{H}_{\ell^n}]}$ is simply $\displaystyle \# \frac{V_{\ell^n}}{[V_{\ell^n}, H_{\ell^n}]}$. We conclude because this quotient is trivial by Lemma \ref{lem-adelic-good-ell}.
\item 
In the CM case, the conclusion follows from Lemma \ref{lemma:CMVanishing} since $\ell \nmid [\tilde{K}:K]\leq 2$.
\end{itemize}

For all other primes, combining Proposition \ref{bound-F-over-K} and Proposition \ref{prop-quotient-bounded} we get that $B_\ell(N)$ divides $[\tilde{K}:K]\cdot \ell^{2\tilde n_\ell}$ 
and we conclude using Lemma \ref{lem-param-tilde-bound}.
\end{proof}
\end{thm}

\begin{rem}
The proof shows that the inequality $v_\ell(B_\ell(N)) \leq 2n_\ell + 3v_\ell \left([\tilde{K}:K] \right)$ holds for every prime $\ell$ and for every rational point $\alpha \in E(K)$. In other words, for a fixed prime $\ell$ the adelic failure can be bounded independently of the rational point $\alpha$.
\end{rem}

We can finally prove our first Theorem from the introduction:
\begin{proof}[Proof of Theorem \ref{thm:Main}]
By Remark \ref{rem-NM}, Theorem \ref{thm:Main} follows from Theorems \ref{thm:UpperBoundAl} and \ref{thm:UpperBoundBl} by taking $C:=\prod_\ell \ell^{a_\ell+b_\ell}$.
\end{proof}

\begin{rem}\label{rmk:MainTheoremIsEffective}
Theorem \ref{thm:Main} is completely effective, in the following sense: the quantities $a_\ell$ and $b_\ell$ can be computed in terms of $[\tilde{K}:K]$, $n_\ell$, and the divisibility parameter $d$. The integer $d$ can be bounded effectively in terms of the height of $\alpha$ and of standard invariants of the elliptic curve, as showed in Remark \ref{rmk:dIsEffective}. The remaining quantities $[\tilde{K}:K]$ and $n_\ell$ can be bounded effectively in terms of $[K:\Q]$ and of the height of $E$, as shown in \cite{MR3437765}.
\end{rem}

\section{A counterexample in the CM case}\label{sec:CMCounterexample}
We give an example showing that Proposition \ref{prop-max-growth-kummer} does not hold in the CM case when $\ell$ is split in the field of complex multiplication, and that in fact in this case there can be no uniform lower bound on the image of the Kummer representation depending only on the image of the torsion representation, even when $\alpha$ is strongly $\ell$-indivisible.

Let $E/\mathbb{Q}$ be an elliptic curve with complex multiplication over $\overline \Q$ by the imaginary quadratic field $F$. Let $\alpha \in E(\mathbb{Q})$ be such that the $\ell^n$-arboreal representation attached to $(E,\alpha)$ maps onto $\left( \Z/\ell^n\Z \right)^2 \rtimes N_{\ell^n}$ for every $n \geq 1$, where $N_{\ell^n}$ is the normaliser of a Cartan subgroup $C_{\ell^n}$ of $\GL_2(\Z/\ell^n\Z)$. Suppose furthermore that $\ell$ is split in $F$ and does not divide the conductor of the order $\operatorname{End}_{\overline{\Q}}{E} \subseteq \mathcal{O}_F$.
Such triples $(E, \alpha, \ell)$ exist: by \cite[Example 5.11]{JonesRouse} we can take $E : y^2 = x^3 + 3x$ (which has CM by $\mathbb{Z}[i]$), $\alpha = (1, -2)$ and $\ell = 5$ (which splits in $\mathbb{Z}[i]$). Notice that for this elliptic curve and this $\alpha$ the same property holds for every $\ell \equiv 1 \pmod 4$: \cite[Theorem 1.5 (2)]{MR3766118} implies that for all $\ell \geq 5$ the image of the Galois representation is the full normaliser of a Cartan subgroup, at which point surjectivity of the Kummer representation follows from \cite[Theorem 5.8]{JonesRouse}.

Consider now the image of the arboreal representation associated with the triple $(E / F, \alpha, \ell)$. Base-changing $E$ to $F$ has the effect of replacing the normaliser of the Cartan subgroup with Cartan itself: more precisely we have $\omega_{\ell^n}\left( \gal( F_{\ell^n, \ell^n} \mid F ) \right) = \left( \Z/\ell^n\Z \right)^2 \rtimes C_{\ell^n}$ for every $n \geq 1$. As $\ell$ is split in the quadratic ring $\operatorname{End}_{\overline{\Q}}(E)$, so is the Cartan subgroup $C_{\ell^n}$, and therefore we can assume -- choosing a different basis for $E[\ell^n]$ if necessary -- that $C_{\ell^n}$ is the subgroup of diagonal matrices in $\GL_2(\Z/\ell^n\Z)$. Fix now a large $n$ and let
\[
B_{\ell^n} = \left\{ (t,M) \in \left( \Z/\ell^n\Z \right)^2 \rtimes C_{\ell^n} : t \equiv (*,0) \pmod{\ell^{n-1}} \right\}.
\]
Using the explicit group law on $(\Z/\ell^n\Z)^2 \rtimes C_{\ell^n}$ one checks without difficulty that $B_{\ell^n}$ is a subgroup of $\left( \Z/\ell^n\Z \right)^2 \rtimes C_{\ell^n}$: indeed, given two elements $g_1=(t_1,M_1)$ and $g_2=(t_2,M_2)$ in $B_{\ell^n}$, we have
\[
g_1 \cdot g_2 = (t_1, M_1) \cdot (t_2, M_2) = (t_1 + M_1 t_2, M_1M_2),
\]
and (since $M_1$ is diagonal) the second coordinate of $t_1 + M_1 t_2$ is a linear combination (with $\Z/\ell^n\Z$-coefficients) of the second coordinates of $t_1, t_2$, hence is zero modulo $\ell^{n-1}$. Finally, let $K \subset F_{\ell^n, \ell^n}$ be the field corresponding by Galois theory to the subgroup $B_{\ell^n}$ of $\left( \Z/\ell^n\Z \right)^2 \rtimes C_{\ell^n} \cong \gal(F_{\ell^n,\ell^n} \mid F)$. 

We now study the situation of Proposition \ref{prop-max-growth-kummer} for the elliptic curve $E/K$ and the point $\alpha$. 
By construction, the image of the $\ell^{n-1}$-torsion representation attached to $(E/K, \ell)$ is $C_{\ell^{n-1}}$, so the parameter of maximal growth can be taken to be $n_\ell=1$. We claim that $\alpha \in E(K)$ is strongly $\ell$-indivisible. The modulo-$\ell$ torsion representation is surjective onto $C_\ell$, so that in particular no $\ell$-torsion point of $E$ is defined over $K$, and strongly $\ell$-indivisible is equivalent to $\ell$-indivisible. To see that this last condition holds, notice that if $\alpha$ were $\ell$-divisible then we would have $K_{\ell, \ell}=K_\ell$. However this is not the case, because by construction $\gal(K_{\ell,\ell} \mid K_\ell) = \{ t \in (\Z/\ell\Z)^2 : t \equiv (*,0) \pmod{\ell} \}$ has order $\ell$. Finally, for $k=n-3$ we have
\[
V_{\ell^{k+2n_\ell}} = V_{\ell^{n-1}} = \{ t \in (\Z/\ell^{n-1}\Z)^2 : t \equiv (*,0) \pmod{\ell^{n-1}} \},
\]
which is very far from containing $E[\ell^k]$ -- in fact, the index of $V_{\ell^{k+2n_\ell}}$ in $E[\ell^{k+2n_\ell}]$ can be made arbitrarily large by choosing larger and larger values of $n$. Notice that in any such example the $\ell$-adic representation will be surjective onto a split Cartan subgroup of $\GL_2(\Z_\ell)$.

\section{Uniform Bounds for the Adelic Kummer Representation}
\label{sec:UniformBounds}

Our aim in this section is to show:

\begin{thm}\label{thm:QUniformity}
There is a positive integer $C$ with the following property: for every elliptic curve $E/\Q$ and every strongly indivisible point $\alpha \in E(\Q)$, the image $W_\infty \cong \gal(K_{\infty, \infty} \mid K_\infty)$ of the Kummer map associated with $(E/\Q,\alpha)$ has index dividing $C$ in $\prod_\ell T_\ell(E)$.
\end{thm}

This result immediately implies Theorem \ref{thm:UniformIntroduction}:
\begin{proof}[Proof of Theorem \ref{thm:UniformIntroduction}]
By Remark \ref{rem:Vsubgroup}, for every $N\mid M$ the ratio
$
\displaystyle \frac{N^2}{[\Q_{M,N}: \Q_M]}$ divides 
\[
\displaystyle \frac{N^2}{[\Q_{\infty,N}: \Q_\infty]}=
\left[(\hat \Z/N\hat\Z)^2:W_\infty/N W_\infty\right],
\]
which in turn divides $[\hat\Z^2:W_\infty]$.
\end{proof}

As in Subsection \ref{subsec:PossibleImagesModl}, we will denote by $\mathcal{T}_0$ the finite set of primes
\[
\mathcal{T}_0:=\set{p\text{ prime }\mid p\leq 17}\cup\{37\}.
\]

\subsection{Bounds on Cohomology Groups}

Let $E/\Q$ be an elliptic curve and $N_1,N_2$ be positive integers with $N_1\mid N_2$. The first step in the proof of Theorem \ref{thm:QUniformity} is to bound the exponent of the cohomology group $H^1(H_{N_2},E[N_1])$. In the course of the proof we shall need the following technical result, which will be proved in Section \ref{section:ProofOf}.

\begin{prop}\label{prop:NewBoundCohomology}
There is a universal constant $e$ satisfying the following property. Let $E/\Q$ be a non-CM elliptic curve, $N$ a positive integer and $\ell$ a prime factor of $N$. Let $\ell^k$ be the largest power of $\ell$ dividing $N$ and $\J=\gal(\Q_N \mid \Q_{\ell^k})\triangleleft H_N$. Consider the action of $H_N$ on $\Hom(J,E[\ell^k])$ defined by $(h\psi)(x)=h\psi(h^{-1}xh)$ for all $h\in H_N$, $\psi:\J\to E[\ell^k]$ and $x\in \J$. Then the exponent of $\operatorname{Hom}\left(\J , E[\ell^k] \right)^{H_N}$ divides $e$.
\end{prop}

\begin{prop}\label{prop:UniformBoundOnCohomology}
There is a positive integer $C_1$ with the following property. Let $E/\mathbb{Q}$ be an elliptic curve, $N_1$ and $N_2$ be positive integers with $N_1\mid N_2$. Then the exponent of $H^1(H_{N_2}, E[N_1])$ divides $C_1$.
\end{prop}

\begin{proof}
We can prove the statement separately for CM and non-CM curves, and then conclude by taking the least common multiple of the two constants obtained in the two cases.

Assume first that $E/\Q$ has CM over $\overline \Q$. Let $F$ be the CM field of $E$, $\mathcal{O}_F$ the ring of integers of $F$ and $\mathcal{O}_\ell:=\mathcal{O}_F \otimes_{\Z}\Z_\ell$. By \cite[Theorem 1.5]{MR3766118} we have $d:=\left[\prod_\ell \mathcal{O}_\ell^\times:H_\infty\cap \prod_\ell \mathcal{O}_\ell^\times\right]\leq 6$. In particular all the $d$-th powers of elements in $\prod_\ell \mathcal{O}_\ell^\times$ are in $H_\infty$, hence we have $\hat \Z^{\times d}\subseteq H_\infty\subseteq \prod_\ell \GL_2(\Z_\ell)$ and $H_\infty$ contains the nontrivial homothety $\lambda=(\lambda_\ell)$, where $\lambda_2=3^d$ and $\lambda_\ell=2^d$ for $\ell\neq 2$. By Sah's Lemma \cite[Lemma A.2]{MR2018998} we have $(\lambda-1)H^1(H_{N_2}, E[N_1])=0$. Notice that the image of $\lambda-1$ in $\Z_\ell$ is nonzero for all $\ell$, and that it is invertible for almost all $\ell$. The claim follows from the fact that $d$ is bounded.

Assume now that $E$ does not have complex multiplication over $\overline \Q$. As cohomology commutes with finite direct products we have
\[
\begin{aligned}
H^1(H_{N_2}, E[N_1]) & \cong H^1 \left( H_{N_2}, \bigoplus_{\ell^v \mid {N_1}} E[\ell^v] \right) \cong \bigoplus_{\ell^v \mid {N_1}}  H^1 \left( H_{N_2}, E[\ell^v] \right).
\end{aligned}
\]
Fix an $\ell$ in this sum and let $\J = \gal(\Q_{N_2} \mid \Q_{\ell^k}) \triangleleft H_{N_2}$, where $\ell^k$ is the largest power of $\ell$ dividing ${N_2}$. By the inflation-restriction sequence we get
\[
0 \to H^1(H_{N_2}/\J, E[\ell^v]^{\J}) \to H^1(H_{N_2}, E[\ell^v]) \to H^1(\J, E[\ell^v])^{H_{N_2}};
\]
since by definition $\J$ fixes $E[\ell^v]$, this is the same as
\[
0 \to H^1( H_{\ell^k} , E[\ell^v]) \to H^1(H_{N_2}, E[\ell^v]) \to \operatorname{Hom}(\J, E[\ell^v])^{H_{N_2}}.
\]
It is clear that the exponent of $H^1(H_{N_2}, E[{N_1}])$ is the least common multiple of the exponents of the direct summands $H^1 \left(H_{N_2}, E[\ell^v] \right)$ for $\ell \mid {N_1}$, so we can focus on one such summand at a time. Furthermore, the above inflation-restriction exact sequence shows that the exponent of $H^1(H_{N_2},E[\ell^v])$ divides the product of the exponents of $H^1( H_{\ell^k} , E[\ell^v])$ and of $\operatorname{Hom}(\J, E[\ell^v])^{H_{N_2}}$. It is enough to give a uniform bound for the exponents of these two cohomology groups.

\begin{itemize}
\item $\boxed{H^1( H_{\ell^k}, E[\ell^v])}$ Assume first that $\ell\not\in\mathcal T_0$. By Theorem \ref{thm:Mazur}, $H_\ell$ is not contained in a Borel subgroup of $\GL_2(\F_\ell)$, so by \cite[Lemma 4]{LawsonWutrich} it contains a nontrivial homothety. By Lemma \ref{lem:LiftHomothety} the image $H_{\ell^\infty}$ of the $\ell$-adic representation contains a homothety that is non-trivial modulo $\ell$, so by Sah's Lemma \cite[Lemma A.2]{MR2018998} we have $H^1( H_{\ell^k}, E[\ell^v])=0$.
For $\ell \in \mathcal{T}_0$ let $n_\ell$ be a universal bound on the parameter of maximal growth of the $\ell$-adic representation, as in Corollary \ref{cor:UniversalBoundGrowthParameter}. 
By Lemma \ref{lemma:ImageAlwaysContainsScalars} we have $(1+\ell^{n_\ell})\Id \in H_{\ell^k}$, and from Lemma \ref{lemma:ExponentOfH1} we obtain that the exponent of $H^1( H_{\ell^k}, E[\ell^v])$ divides $\ell^{n_\ell}$. 
\item $\boxed{\operatorname{Hom}(\J, E[\ell^v])^{H_{N_2}}}$ As $v \leq k$, this group is contained in ${\operatorname{Hom}(\J, E[\ell^k])^{H_{N_2}}}$, whose exponent is uniformly bounded by Proposition \ref{prop:NewBoundCohomology}. Notice that the action of $H_{N_2}$ on $\operatorname{Hom}(J, E[\ell^k])$ is precisely the one considered in Proposition \ref{prop:NewBoundCohomology} by well-known properties of the inflation-restriction exact sequence (see for example \cite[Theorem 4.1.20]{MR1282290}).
\end{itemize}
\end{proof}

\begin{cor}\label{cor:UniformBoundDivisibility}
Let $C_1$ be as in Proposition \ref{prop:UniformBoundOnCohomology}. Let $E/\mathbb{Q}$ be an elliptic curve and let $\alpha \in E(\Q)$ be a strongly indivisible point.
If $\alpha$ is divisible by $n\geq 1$ over $\Q_\infty$, then $n \mid C_1$.
\end{cor}
\begin{proof}
Without loss of generality we can assume that $n=\ell^e$ is a power of a prime $\ell$. Since $\Q_\infty$ is the union of the torsion fields $\Q_N$, there exists $N$ such that $\alpha$ is divisible by $\ell^e$ over $\Q_N$, and we may assume that $\ell^e$ divides $N$. The claim then follows from Lemma \ref{lem-not-divisible}, since by Proposition \ref{prop:UniformBoundOnCohomology} the exponent of $H^1(\gal(\Q_N\mid\Q),E[\ell^e])$ is a power of $\ell$ that divides $C_1$.
\end{proof}

\begin{lem}\label{cor:BoundInTermsOfC1}
Let $C_1$ be as in Proposition \ref{prop:UniformBoundOnCohomology}. The following hold for every prime $\ell$:
\begin{enumerate}
\item The $\Z_\ell$-module $W_{\ell^\infty}$, considered as a submodule of $\Z_\ell^2$, contains a vector of valuation at most $v_\ell(C_1)$.
\item If $E$ does not have CM over $\overline\Q$ and $n_\ell$ is a parameter of maximal growth for the $\ell$-adic torsion representation, then $W_{\ell^\infty}$ contains $\ell^{n_\ell + v_\ell(C_1)} T_\ell(E)$.
\item If $E[\ell]$ is an irreducible $H_\ell$-module, then $W_{\ell^\infty}$ contains $\ell^{v_\ell(C_1)} T_\ell(E)$.
\item If $E$ has CM over $\overline \Q$, let $(\gamma,\delta)$ be parameters for the Cartan subgroup of $\GL_2(\Z_\ell)$ corresponding to $\operatorname{End}_{\overline \Q}(E)$. If $n_\ell$ is a parameter of maximal growth for the $\ell$-adic torsion representation, then $W_{\ell^\infty}$ contains $\ell^{3n_\ell + v_\ell(4\delta C_1)} T_\ell(E)$.
\end{enumerate}
\end{lem}
\begin{proof}
Part (1) follows from Lemma \ref{lem-order-n-d}, since by Corollary \ref{cor:UniformBoundDivisibility} the point $\alpha$ is not divisible by $\ell^{v_\ell(C_1)+1}$ over $\Q_\infty$. Parts (2), (3) and (4) then follow from Proposition \ref{prop:GroupTheory} (for part (4) observe that no elliptic curve over $\Q$ has CM defined over $\Q$).
\end{proof}

We can now prove the main Theorem of this section.

\begin{proof}[Proof of Theorem \ref{thm:QUniformity}]
As already explained, we have $W_\infty = \prod_\ell W_{\ell^\infty}$, so we obtain
\[
\left[\prod_\ell T_\ell(E) : W_\infty\right] = \prod_\ell [T_\ell(E) : W_{\ell^\infty}].
\]

 Let 
\[\mathcal T_1= \mathcal T_0\cup\{\ell\text{ prime}\mid\ell \text{ divides }C_1\}\cup \set{19,43,67,163}.\]
Notice that by Theorem \ref{thm:Mazur} for $\ell\not\in \mathcal T_1$ there is no elliptic curve over $\Q$ with a rational subgroup of order $\ell$. By Lemma \ref{cor:BoundInTermsOfC1} (3), for $\ell\not\in \mathcal T_1$ we have $W_{\ell^{\infty}}=T_\ell(E)$, so
\begin{equation}
\label{eqn:boundIndexKummerUniform}
\left[\prod_\ell T_\ell(E) : W_\infty\right] = \prod_{\ell \in \mathcal T_1} [T_\ell(E) : W_{\ell^\infty}].
\end{equation}

Now it is enough to prove the Theorem separately in the CM and in the non-CM case, and then take the least common multiple of the two constants obtained.

Suppose first that $E$ does not have CM over $\overline{\Q}$. Applying Lemma \ref{cor:BoundInTermsOfC1}(2) we see that $[T_\ell(E) : W_{\ell^\infty}]$ divides $\ell^{2(n_\ell+v_\ell(C_1))}$, where $n_\ell$ is a parameter of maximal growth for the $\ell$-adic torsion for $E$. By Theorem \ref{thm:Arai} this can be bounded uniformly in $E$. Since $C_1$ does not depend on $E$, each factor of the right hand side of \eqref{eqn:boundIndexKummerUniform} is uniformly bounded.

Assume now that $E$ has complex multiplication over $\overline \Q$ and let $(\gamma,\delta)$ be parameters for the Cartan subgroup of $\GL_2(\Z_\ell)$ corresponding to $\operatorname{End}_{\overline \Q}(E)$. Applying Lemma \ref{cor:BoundInTermsOfC1}(4), we see that $[T_\ell(E) : W_{\ell^\infty}]$ divides $\ell^{2(3n_\ell+v_\ell(4\delta C_1))}$, where $n_\ell$ is a parameter of maximal growth for the $\ell$-adic torsion representation for $E$, which is uniformly bounded by Corollary \ref{cor:UniversalBoundGrowthParameter}. It remains to show that $v_\ell(\delta)$ can be bounded uniformly as well. This follows from the fact that $\delta$ only depends on the $\overline \Q$-isomorphism class of $E$, and that there are only finitely many rational $j$-invariants corresponding to CM elliptic curves.
\end{proof}

\subsection{Proof of Proposition \ref{prop:NewBoundCohomology}}
\label{section:ProofOf}

Recall the setting of Proposition \ref{prop:NewBoundCohomology}: $E/\Q$ is a non-CM elliptic curve, $N$ is a positive integer, and $\ell$ is a prime factor of $N$. Let $\ell^k$ be the largest power of $\ell$ dividing $N$ and $\J=\gal(\Q_N \mid \Q_{\ell^k})\triangleleft H_N$. The question is to study the exponent of the group $\operatorname{Hom}\left(\J , E[\ell^k] \right)^{H_N}$. In order to do this, we shall study the conjugation action of $g \in H_N$ on the abelianisation of $\J$.
More generally, we shall also consider the conjugation action of elements in $\GL_2(\Z/N\Z)$ that normalise $\J$.

It will be useful to work with a certain subgroup $\JJ$ of $\J$. More generally, we introduce the following notation.

\begin{defi}
Let $G$ be a group and $M$ a positive integer. We denote by $G(M)$ the subgroup of $G$ generated by $\set{g^M\mid g\in G}$.
\end{defi}

\begin{lem}\label{lem:SubgroupL}
The subgroup $\JJ$ is normal in $\J$, the quotient group $\J/\JJ$ has exponent at most 2, $\JJ$ is stable under the conjugation action of $H_N$, and
\[
\operatorname{exp} \operatorname{Hom}\left(\J, E[\ell^k] \right)^{H_N} \mid 2 \operatorname{exp} \operatorname{Hom}\left(\JJ, E[\ell^k] \right)^{H_N}.
\]
\end{lem}
\begin{proof}
Clearly $\JJ$ is a characteristic subgroup of $\J$, so it is normal in $\J$ and stable under the conjugation action of $H_N$ on $\J$. Given a coset $h\JJ \in \J/\JJ$ we have $(h\JJ)^2 = h^2\JJ = \JJ$ since $h^2 \in \JJ$ by definition, so the quotient $\J/\JJ$ is killed by $2$. Finally, take a homomorphism $\psi : \J \to E[\ell^k]$ stable under the conjugation action of $H_N$ and denote by $d$ the exponent of the abelian group $\operatorname{Hom}\left(\JJ, E[\ell^k] \right)^{H_N}$. The restriction of $\psi$ to $\JJ$ is an element of $\operatorname{Hom}\left(\JJ, E[\ell^k] \right)^{H_N}$, so it satisfies $d\psi|_{\JJ} = 0$, and thus given any $h \in \J$ we have $d\psi|_{\JJ}(h^2)=0$. This implies that for every $h \in J$ we have $2d \psi(h) = 0$, hence $\psi$ is killed by $2d$. Since this is true for all $\psi$, the claim follows.
\end{proof}

We will also need the following two simple lemmas:

\begin{lem}\label{lem:NontrivialHomothety}
Let $E/\Q$ be an elliptic curve and let $M\geq 37$ be an integer. If $\ell > M+1$ is a prime number, then $H_{\ell^\infty}{(M)}$ contains a homothety $\lambda \operatorname{Id}$ with $\lambda \not \equiv 1 \pmod{\ell}$.
\end{lem}
\begin{proof}
By Corollary \ref{cor:ContainsScalarsAndConjugation}, since $\ell>M+1>37$, the image of the modulo-$\ell$ representation contains all the homotheties. In particular, if $\overline \mu \in \F_\ell^\times$ is a generator of the multiplicative group $\F_\ell^\times$, then $H_\ell$ contains $\overline \mu \operatorname{Id}$, so by Lemma \ref{lem:LiftHomothety} $H_{\ell^\infty}$ contains $\mu\Id$, where $\mu\in\Z_\ell^\times$ is congruent to $\overline{\mu}$ modulo $\ell$.
So $H_{\ell^\infty}{(M)}$ contains $\mu^{M} \operatorname{Id}$, which is nontrivial modulo $\ell$ since $\overline \mu$ has order $\ell-1 > M$. 
\end{proof}

\begin{lem}
\label{lem:Cohen}
Let $p$ be a prime and let $n$ be a positive integer (with $n\geq 2 $ if $p=2$). For every positive integer $k$ let $U_k=\set{ x\in\Z_p\mid x\equiv 1 \pmod {p^k}}$. Let $M$ be a positive integer. Then $\set{x^M\mid x\in U_n}\supseteq U_{n+v_p(M)}$.
\end{lem}
\begin{proof}
Let $y\in U_{n+v_p(M)}$ and let $a=y-1$. By \cite[Corollary 4.2.17 and Corollary 4.2.18(1)]{MR2312337}, the $p$-adic integer $x=\exp(M^{-1}\log y)$ is well defined and satisfies $v_p(x-1)\geq v_p(M^{-1}a)\geq n$. Therefore $x\in U_n$ and clearly $x^M=y$.
\end{proof}

We will derive Proposition \ref{prop:NewBoundCohomology} from the following statement:
\begin{prop}\label{prop:BoundExponentg}
There is a universal constant $M$ with the following property. For every elliptic curve $E/\Q$, every positive integer $N$, every prime power $\ell^k$ dividing $N$, and every $g \in H_N$, the conjugation action of $g^M$ on the abelianisation of $\JJ$ is trivial.
\end{prop}
\begin{proof}[Proof that Proposition \ref{prop:BoundExponentg} implies Proposition \ref{prop:NewBoundCohomology}]
By Lemma \ref{lem:SubgroupL} it suffices to prove Proposition \ref{prop:NewBoundCohomology} with $\J$ replaced by $\JJ$. Let $\psi \in \operatorname{Hom}\left(\JJ, E[\ell^k] \right)$: then as $E[\ell^k]$ is abelian $\psi$ factors through $\JJ^{\ab}$.

For every $g \in H_N$, every $\psi \in \operatorname{Hom}\left(\JJ, E[\ell^k] \right)^{H_N}$ and every $h\in \JJ$ we have
\[
\psi(h)=g^M \cdot \psi(g^{-M} h g^M ) =g^M \cdot \psi(h),
\]
where the first equality holds because $\psi$ is $H_N$-invariant and the second  because the automorphism induced by $g^M$ on $\JJ^{\operatorname{ab}}$ is trivial by Proposition \ref{prop:BoundExponentg}. This means that the image of $\psi$ is contained in $E[\ell^k]^{H_N{(M)}}$. Since the action of $H_N$ on $E[\ell^k]$ factors via the canonical projection $H_N \to \GL_2(\Z/\ell^k\Z)$, this is the same as saying that the image of $\psi$ is contained in the subgroup of $E[\ell^k]$ fixed under $H_{\ell^k}(M)$. It remains to show that the exponent of $E[\ell^k]^{H_{\ell^k}(M)}$ is uniformly bounded, and trivial for $\ell$ sufficiently large.

To see this, recall that by Theorem \ref{thm:Arai} there exists an integer $n\geq 1$, independent of $E$, such that $H_{\ell^k}$ contains $\operatorname{Id} + \ell^{n} \Mat_2(\Z/\ell^k\Z)$ (and we have $n \geq 2$ if $\ell=2$). By Lemma \ref{lem:Cohen}, for every $E/\Q$ the group $H_{\ell^k}(M)$ contains all scalar matrices in $\Mat_2(\Z/\ell^k\Z)$ that are congruent to the identity modulo $\ell^{n+v_\ell(M)}$.
We claim that the exponent of $E[\ell^k]^{H_{\ell^k}(M)}$ divides $\ell^{n+v_\ell(M)}$. In fact, by what we have seen $H_{\ell^k}(M)$ contains $(1+\ell^{n+v_\ell(M)})\Id$, so $E[\ell^k]^{H_{\ell^k}(M)}$ is in particular fixed by $(1+\ell^{n+v_\ell(M)})\Id$, hence it is contained $E[\ell^{n+v_\ell(M)}]$.

Finally, we show that $\Hom(\J,E[\ell^k])^{H_N}$ is trivial for $\ell>M+1$. Since $\ell>2$, by Lemma \ref{lem:SubgroupL} it is enough to show that $\Hom(\JJ,E[\ell^k])^{H_N}$ is trivial. As above, the image of any $H_N$-stable homomorphism from $\JJ$ to $E[\ell^k]$ is contained in the $H_{\ell^k}(M)$-fixed points of $E[\ell^k]$. By Lemma \ref{lem:NontrivialHomothety}, $H_{\ell^k}(M)$ contains a homothety which is nontrivial modulo $\ell$, so we are done since the only fixed point of this homothety is $0$.
\end{proof}

We now turn to the proof of Proposition \ref{prop:BoundExponentg}.
We start by showing that we may assume $N$ to be of the form $\ell^k \cdot \prod_{p \mid N, p \neq \ell} p$. To see this, let $N=\ell^k \prod_{p \mid N, p \neq \ell} p^{e_p}$ be arbitrary and let $N':=\ell^k \prod_{p \mid N, p \neq \ell} p$. There is an obvious reduction map $\J \to \gal(\Q_{N'} \mid \Q_{\ell^k})$. The kernel $\mathcal K$ of this map is a subgroup of $\J$ whose order is divisible only by primes $p \mid N, p \neq \ell$. Recall that we will be considering $\operatorname{Hom}(\J, E[\ell^k])^{H_N}$. Let $\psi : \J \to E[\ell^k]$ be a homomorphism: we claim that $\psi$ factors via the quotient $\gal(\Q_{N'} \mid \Q_{\ell^k})$. Indeed, all the elements in $\mathcal K$ have order prime to $\ell$, hence they must go to zero in $E[\ell^k]$. Therefore we may assume $N=N'$, that is, $N=\ell^k \cdot \prod_{p \mid N, p \neq \ell} p$.

We identify $H_N$ with a subgroup of $\GL_2(\Z/\ell^k\Z) \times \prod_{p \mid N, p \neq \ell} \GL_2(\Z/p\Z)$ and $\J$ with the subgroup of $H_N$ consisting of elements having trivial first coordinate, and for $g\in H_N$ we write $g=(g_\ell, g_{p_1}, \ldots, g_{p_r})$ with $g_\ell \in \GL_2(\Z/\ell^k\Z)$ and $g_{p_i} \in \GL_2(\Z/p_i\Z)$. Finally, for $p\mid N$, $p\neq \ell$ we denote by $\pi_{p_i}:H_N\to \GL_2(\Z/p_i\Z)$ the projection on the factor corresponding to $p_i$, and we denote by $\pi_\ell:H_N\to \GL_2(\Z/\ell^k\Z)$ the projection on the factor corresponding to $\ell$.

\begin{lem}\label{lem:SL2Factor}
Let $p$ be a prime factor of $N$ with $p \geq 7$, $p \neq \ell$. Suppose that the modulo-$p$ representation attached to $E/\Q$ is surjective. Then $\JJ$ contains $\{1\} \times \cdots \times \{1\} \times \SL_2(\Z/p\Z) \times \{1\} \times \cdots \times \{1\}$.
\end{lem}
\begin{proof}
Clearly $\PSL_2(\F_p)$ occurs in $H_N$. Hence it must occur either in $\J$ or in $H_N/\J$, but the latter is isomorphic to a subgroup of $\GL_2(\Z/\ell^k\Z)$ with $\ell \neq p$, so it must occur in $\J$. Consider the kernel of the projection $\J \to \prod_{q \mid N, q \neq p} \GL_2(\Z/q\Z)$: then $\PSL_2(\F_p)$ must occur either in this kernel or in $\prod_{q \mid N, q \neq p} \GL_2(\Z/q\Z)$, but the latter case is impossible. Using Lemma \ref{lemma-solvable}, it follows immediately that $\J$ contains $\{1\} \times \cdots \times \{1\} \times \SL_2(\Z/p\Z) \times \{1\} \times \cdots \times \{1\}$. We conclude by noting that $\SL_2(\F_p)$ is generated by its squares.
\end{proof}

\begin{lem}\label{lem:SameAutomorphism}
Let $g \in H_N$ and $h \in \JJ$. Then $gh \in H_N$, and the automorphisms of $\JJ^{\operatorname{ab}}$ induced by $g$ and by $gh$ coincide.
\end{lem}
\begin{proof}
As $\JJ$ is a subgroup of $H_N$, the fact that $gh \in H_N$ is obvious. For the second statement, notice that for every $x \in \JJ$ the element
$(gh)^{-1} x (gh)$
differs from $g^{-1}xg$ by multiplication by $h^{-1}(g^{-1}x^{-1}g)^{-1}h(g^{-1}x^{-1}g)$, which is a commutator in $\JJ$. Hence the classes of $(gh)^{-1} x (gh)$ and $g^{-1}xg$ are equal in $\JJ^{\operatorname{ab}}$.
\end{proof}

\begin{lem}\label{lem:gpNormalisesKp}
For each $p \mid N, p \neq \ell$, the component $g_p$ of $g$ along $\GL_2(\Z/p\Z)$ normalises $\pi_{p}(\JJ)$ in $\GL_2(\Z/p\Z)$.
\end{lem}
\begin{proof}
Since $H_N$ normalises $\JJ$ by Lemma \ref{lem:SubgroupL}, we have $\pi_p(g^{-1}\JJ g)=\pi_p(\JJ)$. On the other hand $\pi_p(g^{-1}\JJ g) = \pi_p(g)^{-1} \pi_p(\JJ) \pi_p(g)$, so $g_p^{-1} \pi_p(\JJ) g_p=\pi_p(\JJ)$ as desired.
\end{proof}
\begin{cor}\label{cor:gpId}
Let $p_1,\dots,p_s \geq 7$ be primes all different from $\ell$ and such that the mod-$p_i$ representation attached to $E/\Q$ is surjective for each $p_i$. Let $g\in H_N$ and let $\hat g$ be the element of $\GL_2(\Z/N\Z)$ obtained by replacing every $p_i$-component (for $i=1,\ldots,s$) of $g$ by $\Id$. Then $\hat g^2$ normalises $\JJ$, and it induces on $\JJ^{\ab}$ the same conjugation action as $g^2$.
\end{cor}
\begin{proof}
By Lemma \ref{lem:SameAutomorphism}, if we multiply $g^2$ by any element of $\JJ$ the conjugation action on $\JJ^{\operatorname{ab}}$ does not change. By construction, the determinant of $\pi_{p_i}(g^2)=g_{p_i}^2$ is a square in $\F_{p_i}^\times$, say $\lambda_i^2$. It follows that the determinant of $g_{p_i}^2/\lambda_i$ is 1, so $g_{p_i}^2/\lambda_i \in \SL_2(\Z/p_i\Z)$.
By Lemma \ref{lem:SL2Factor} we have that $\JJ$ contains $h_i=(1,1,\ldots,1,g_{p_i}^2/\lambda_i,1,\ldots,1)$. Letting $h=h_1\cdots h_s$, we obtain that the action of $g^2h^{-1}$ is the same as that of $g^2$. But the element
\[
\mu=(1,\dots,1,\lambda_1,1,\dots,1)\cdots (1,\dots,1,\lambda_s,1,\dots,1)
\]
 is central in $\GL_2(\Z/N\Z)$, so $\hat g^2=g^2h^{-1}\mu^{-1}$ normalises $\JJ$ and it induces the same action as $g^2$ on $J(2)^{\ab}$.
\end{proof}

Let $M=\operatorname{lcm} \{\exp \PGL_2(\F_p) : p \in \mathcal T_0\}$, where $\exp \PGL_2(\F_p)$ denotes the exponent of the group $\PGL_2(\F_p)$.

\begin{rem}
\label{rem:evenScalar}
 Notice that $M$ is even. Moreover, for any $g\in \GL_2(\Z/N\Z)$ and any $p\in \mathcal T_0$ with $p\mid N$ and $p\neq \ell$ we have that $\pi_p(g^M)$ is a scalar in $\GL_2(\F_p)$, since it is trivial in $\PGL_2(\F_p)$.
\end{rem}

We now prove Proposition \ref{prop:BoundExponentg}, using the constant $M$ just introduced.

\begin{proof}[Proof of Proposition \ref{prop:BoundExponentg}]
Write as before $g=(g_p)$. We divide the prime factors of $N$ different from $\ell$ into three sets as follows:
\begin{align*}
\mathcal{P}_0&=\set{p\mid N\text{ such that } p\in\mathcal{T}_0,\,p\neq \ell},\\
\mathcal{P}_1&=\set{p\mid N\text{ such that } H_p=\GL_2(\F_p),\,p\neq \ell},\\
\mathcal{P}_2&=\set{p\mid N\text{ such that } H_p\text{ is conjugate to a subgroup of }N_{\operatorname{ns}}(p),\,p\neq \ell}.
\end{align*}
Notice that by Theorem \ref{thm:Zywyna} each prime factor of $N$ different from $\ell$ belongs to one of these three sets.

We now apply Corollary \ref{cor:gpId} with $\{p_1,\dots,p_s\}=\mathcal{P}_1$ to obtain an element $\hat g\in \GL_2(\Z/N\Z)$ such that $\pi_{p}(\hat g)=\Id$ for every $p\in\mathcal{P}_1$ and such that $\hat{g}^2$ induces on $\JJ^{\ab}$ the same conjugation action as $g^2$. In particular, $\hat g^M$ induces on $\JJ^{\ab}$ the same conjugation as $g^M$ (recall that $M$ is even).

We now prove that this conjugation action is trivial by showing that $\hat g^M$ commutes with every element of $\JJ$. It suffices to show that for each $p\mid N$ the projection $\pi_p(\hat g^M)$ commutes with every element of $\pi_p(\JJ)$.

\begin{itemize}
\item \emph{Case $p\in \mathcal{P}_0$:} by Remark \ref{rem:evenScalar}, $\pi_p(\hat g^M)$ is a scalar, thus it commutes with all of $\GL_2(\F_p)$.
\item \emph{Case $p\in \mathcal{P}_1$:} by construction $\pi_p(\hat g^M)$ is trivial.
\item \emph{Case $p\in \mathcal{P}_2$:} by Corollary \ref{cor:ContainsScalarsAndConjugation} applied to $\pi_p(\hat g)$, there is $h\in \GL_2(\F_p)$ such that $\pi_p(\hat g)\in hN_{\operatorname{ns}}(p)h^{-1}$ and $H_p\subseteq hN_{\operatorname{ns}}(p)h^{-1}$. Since $M$ is even and $C_{\operatorname{ns}}(p)$ has index $2$ in $N_{\operatorname{ns}}(p)$, $\pi_p(\hat g^M)\in hC_{\operatorname{ns}}(p)h^{-1}$ and $\pi_p(\JJ)\subseteq \langle a^2\mid a\in H_p\rangle\subseteq hC_{\operatorname{ns}}(p)h^{-1}$. Since $C_{\operatorname{ns}}(p)$ is abelian, $\pi_p(\hat g^M)$ commutes with every element of $\pi_p(\JJ)$.
\item \emph{Case $p=\ell$:} by construction $\pi_p(J(2))$ is trivial.
\end{itemize}
\end{proof}

\bibliographystyle{acm}
\bibliography{biblio}

\end{document}